\documentclass[12pt]{amsart}
\usepackage{amsmath,amssymb,amsthm,graphics,amscd,mathrsfs}
\usepackage{amscd, amsfonts, amssymb, amsthm}
\usepackage{fullpage, verbatim}
\usepackage[colorlinks, breaklinks, linkcolor=blue]{hyperref}
\usepackage{breakurl}
\usepackage{array}
\usepackage{latexsym}
\usepackage[shortlabels]{enumitem}

\usepackage{xcolor}
\usepackage{subcaption}
\usepackage{tikz-cd}
\usepackage{tikz}
\usepackage{cleveref}
\usetikzlibrary{arrows}
\usetikzlibrary{shapes}
\usepackage{tkz-graph}
\usetikzlibrary{decorations}
\usetikzlibrary{arrows,decorations.markings}
\usepackage{multirow}
\tikzstyle{v} = [circle, draw, inner sep=2pt, minimum size=3pt, fill=black]
\tikzstyle{l} = [rectangle, draw, rounded corners]

\usepackage{caption}
\usetikzlibrary{calc}

\newcommand\CA{{\mathscr A}}
\newcommand\SA{{\mathscr A}}
\newcommand\CB{{\mathscr B}}

\newcommand\CT{{\mathcal T}}

\newcommand\BBC{{\mathbb C}}
\newcommand\BBH{{\mathbb H}}

\newcommand\BBQ{{\mathbb Q}}
\newcommand\BBR{{\mathbb R}}
\newcommand\BBZ{{\mathbb Z}}

\newcommand\codim{\operatorname{codim}}

\newcommand\Fix{{\operatorname{Fix}}}

\newcommand\GL{\operatorname{GL}}

\newcommand\rk{\operatorname{rk}}

\newcommand\Sp{\operatorname{Sp}}

\newcommand\diag{{\operatorname{diag}}}
\newcommand\ord{{\operatorname{ord}}}
\newcommand\Ind{{\operatorname{Ind}}}
\newcommand\Cent{{\operatorname{Cent}}}

\newcommand\inverse{^{-1}}

\newcommand\id{{\operatorname{id}}}

\numberwithin{equation}{section}

\theoremstyle{plain}
\newtheorem{lemma}[equation]{Lemma}
\newtheorem{theorem}[equation]{Theorem}

\newtheorem{corollary}[equation]{Corollary}
\newtheorem{prop}[equation]{Proposition}
\theoremstyle{definition}
\newtheorem{defn}[equation]{Definition}
\newtheorem{remark}[equation]{Remark}

\newtheorem{example}[equation]{Example}

\subjclass[2010]{}

\begin{document}

%%%%%%%%%%%%%%%%%%%%%%%%%%%%%%%%%%%%%%%%%%%%%%%%%%%%%%%%%%%%%%%%%%%%%%
%%%%%%%%%%%%% top matter stuff
%%%%%%%%%%%%%%%%%%%%%%%%%%%%%%%%%%%%%%%%%%%%%%%%%%%%%%%%%%%%%%%%%%%%%%
\title[Invariants in the cohomology of the complement of quaternionic reflection arrangements]
{Invariants in the cohomology of the complement of quaternionic reflection arrangements}
	
		\author[L. Giordani]{Lorenzo Giordani}
\address
{Fakult\"at f\"ur Mathematik,
	Ruhr-Universit\"at Bochum,
	D-44780 Bochum, Germany}
\email{lorenzo.giordani@rub.de}

\author[G. Röhrle]{Gerhard Röhrle}
\address
{Fakult\"at f\"ur Mathematik,
	Ruhr-Universit\"at Bochum,
	D-44780 Bochum, Germany}
\email{gerhard.roehrle@rub.de}

\author[J. Schmitt]{Johannes Schmitt}
\address
{Fakult\"at f\"ur Mathematik,
	Ruhr-Universit\"at Bochum,
	D-44780 Bochum, Germany}
\email{johannes.schmitt@rub.de}

\allowdisplaybreaks
	
	\begin{abstract}
		Let $\CA$ be a hyperplane arrangement in a vector space $V$ and $G\leq\GL(V)$ a group permuting $\CA$. In case when $G$ is a complex reflection group and $\CA=\CA(G)$ its reflection arrangement in $V$, 
		Douglass, Pfeiffer, and Röhrle \cite{douglaspfeifferroehrle:invariants} studied the invariants of the $\BBQ G$-module $H^*(M(\CA);\BBQ)$, the rational, singular cohomology of the complement space $M(\CA)$ in $V$. In this paper we generalize the work in  
		\cite{douglaspfeifferroehrle:invariants} to the case of quaternionic reflection groups, first classified by Cohen in 1980. Surprisingly, only one additional family of new types of Poincar\'e polynomials occurs in the quaternionic setting which is not realised in the complex case, namely those of a particular class of 
		imprimitive irreducible quaternionic reflection groups. We utilize a new description of the lattices of intersections of imprimitive groups as Dowling lattices. Finally, we discuss bases of the space of $G$-invariants in $H^*(M(\CA);\BBQ)$.
	\end{abstract}

\maketitle

\tableofcontents

\section{Introduction}
Frequently, questions 
relating to reflection arrangements $\CA(G)$, where $\CA(G)$ 
consists of the reflecting hyperplanes of an underlying reflection group $G$ in its reflection representation $V$, 
arose first for symmetric groups, i.e.~for braid arrangements, 
then were extended to the remaining
Coxeter groups and finally embraced 
the entire class of complex reflection groups.
A prime example of this phenomenon is the question about the
topological nature of the complement $M(\CA(G))$ of the
union of the hyperplanes of the reflection arrangement $\CA(G)$ in $V$, see \cite{bessis:finite}.
In this paper we traverse steps along  
a similar route concerning questions on the cohomology of the 
complement of a quaternionic reflection arrangement.

Specifically, our aim is to study the representation of a quaternionic reflection group $G$ on the cohomology of the complement $M(\CA(G))$ of its quaternionic reflection arrangement $\CA(G)$, generalizing the study \cite{douglaspfeifferroehrle:invariants}.
In \emph{loc.~cit.}, Douglass, Pfeiffer    
and the second author
refined Brieskorn's study of the cohomology of the complement of a Coxeter arrangement \cite{Bri73} and generalised it to the case of a complex reflection group $G$ (indeed \cite{douglaspfeifferroehrle:invariants} addresses the more general setting of reflection cosets from \cite{brouemallemichel:spetsesI}); see also the references in \cite[\S 1.2]{douglaspfeifferroehrle:invariants} about earlier work by Lehrer and Callegaro--Marin in the latter case.\\

In what follows, let \(\BBH\) be the skew-field of quaternions and let \(V\) be a finite-dimensional right \(\BBH\)-vector space. A 
\emph{quaternionic reflection arrangement} in $V$ is a pair
$(\SA, G)$, where $G$ is a finite subgroup of the general linear group
$\GL(V)$ generated by \emph{quaternionic reflections} (\Cref{def:quatrefl}) and 
$\SA$ is the \emph{quaternionic reflection arrangement} consisting of 
the codimension one (right) \(\BBH\)-subspaces of $V$ fixed by the reflections in $G$ (\Cref{def:quatreflarr}). Thus $G$ acts on $\SA$ and so in turn on the complement of $\SA$ in $V$, 
\[
M(\SA) :=V\setminus \bigcup_{H\in \SA} H.
\]
Clearly, $M(\SA)$ is a $G$-stable complex submanifold of $V$. Let
\[
H^*(M(\SA);\BBQ) = \bigoplus_{k\geq 0} H^k(M(\SA);\BBQ)
\]
denote the rational, singular cohomology of $M(\SA)$.

The rule $g\mapsto (g\inverse)^*$ endows $H^*(M(\SA);\BBQ)$ with the structure of a
graded $\BBQ G$-algebra.
Ultimately, we are interested in the dimensions of the graded components of the ring of \(G\)-invariants $H^{*}(M(\CA);\BBQ)^G$ of $H^*(M(\SA);\BBQ)$.
These dimensions are encoded in the \emph{Poincaré polynomial} of \(H^*(M(\CA);\BBQ)^G\) defined as \[P(\CA, G; t) := \sum_{k \geq 0}\dim_\BBQ(H^k(M(\CA);\BBQ)^G)t^k.\]
Note that if $G$ is a complex reflection group acting on a complex vector space \(V'\) then $G$ can be viewed as a quaternionic reflection group acting on \(V'\otimes_\BBC\BBH\) with corresponding complex and quaternionic reflection arrangements \(\CA_\BBC\) and \(\CA_\BBH\), respectively.
In this case the Poincar\'e polynomial $P(\CA_\BBH,G; t)$ is identical to \(P(\CA_\BBC, G;t)\), modulo a degree shift, see \Cref{thm:poincarecompred}.

We note also that there are canonical reductions to the case of an essential irreducible reflection arrangement, as in \cite[\S 1.11, \S 3]{douglaspfeifferroehrle:invariants}.

In order to determine the Poincar\'e polynomials $P(\SA,G; t)$ above, we argue similarly as in \cite[(2.7), (2.8)]{douglaspfeifferroehrle:invariants} by induction. For the analogue of Brieskorn's Lemma \cite[Lem.~5.91]{orlikterao:arrangements} in 
the quaternionic setting, see \Cref{cBrieskorn}.
This in turn leads to the Euler characteristic identity
\eqref{eq:euler} which allows for an inductive computation of $P(\SA,G; t)$, as in \emph{loc.~cit.}\ by means of \Cref{Gdecomp}.

In case of quaternionic groups $G$ in dimension $2$, the degree $6$ polynomial $P(\SA,G; t)$ is readily obtained from \eqref{eq:euler} and the fact that 
 $H^3(M(\CA);\BBQ)$ affords the permutation representation of $G$  on \(\CA\), see \Cref{prop:dim2}.

For the higher-dimensional (non-complex) irreducible quaternionic reflection groups, we 
analyze the case of an imprimitive reflection group separately in Section \ref{sec:imprimgroups}; see \Cref{thm:poincareimprim} and
\Cref{cor:poincareimprim}. There are precisely seven additional primitive cases to be considered; they are handled in Section \ref{sec:primgroups}, see also  \Cref{tab:primgroups}. 

Finally, in Section \ref{sec:bases} we study sets of canonical bases for \(H^*(M(\CA))^G\), extending \cite[\S 7]{douglaspfeifferroehrle:invariants}.

Surprisingly, only one additional family of new types of Poincar\'e polynomials $P(\SA,G; t)$ occurs in the quaternionic setting which is not already present in the complex case, namely those of a particular class of imprimitive irreducible quaternionic reflection groups, see \Cref{cor:poincareimprim} \labelcref{cor:poincareimprim:2}.
Thus while the groups in the non-complex quaternionic case are rather different from their complex cousins, the resulting Poincar\'e polynomials $P(\SA,G; t)$ are the same with this one exception. 

A similar phenomenon occurs when we consider the Poincar\'e polynomials of 
  the cohomology of the complement of a  quaternionic reflection arrangement.
It is well known that for complex reflection arrangements the Poincar\'e polynomials of the total space  
$H^*(M(\SA);\BBQ)$ factors into linear terms where the exponents of the underlying reflection arrangement feature as the integer roots of the linear factors.
Recently,  S.~Griffeth and D.~Guevara considered the 
quaternionic analogue \cite{griffethguevara}. Astonishingly, in all but a mere three exceptional instances 
all the Poincar\'e polynomials of 
$H^*(M(\SA);\BBQ)$ also factor into linear terms with positive integer roots. It would be desirable to explain this mysterious analogy to the complex case in the quaternionic setting. 

\subsection*{Acknowledgements}
We thank Stephen Griffeth for comments on a preliminary version of this article.

\section{Preliminaries}
In this section, we recall the relevant definitions and results from quaternionic reflection groups and quaternionic arrangements.

\subsection{Quaternionic reflection groups}
\label{sec:quatreflgroups}

Let \(\BBH\) be the skew-field of quaternions. If needed, 
we write \(\{1,\mathbf i,\mathbf j,\mathbf k\}\) for the standard basis of \(\BBH\) over \(\BBR\) and we write \(\{1,\mathbf j\}\) for the standard basis of \(\BBH\) over \(\BBC\).

Let \(V\) be a finite-dimensional right \(\BBH\)-vector space.
Let \(\GL(V)\) be the group of all invertible \(\BBH\)-linear transformations of \(V\).
We agree that \(\GL(V)\) acts on \(V\) from the left.

\begin{defn}
	\label{def:quatrefl}
  An element \(g\in \GL(V)\) of finite order is a \emph{quaternionic reflection} (or just \emph{reflection}), if \(\rk(1 - g) = 1\), that is, \(g\) fixes a subspace of codimension 1 in \(V\).
  A finite group \(G\leq\GL(V)\) is a \emph{quaternionic reflection group} if \(G\) is generated by quaternionic reflections.
\end{defn}

We call a quaternionic reflection group \(G\leq\GL(V)\) \emph{(quaternionic) irreducible}, if there is no \(G\)-invariant decomposition \(V = V_1\oplus V_2\) into right \(\BBH\)-vector spaces with \(V_i \neq \{0\}\).
The irreducible quaternionic reflection groups are classified by Cohen \cite{cohen:quaternionic}.
Waldron \cite{waldron} and Taylor \cite{Tay25} independently revise part of this classification and notably construct further groups in rank 2.
We give a brief overview of the classification.

First of all, notice that complex reflection groups can naturally be considered as quaternionic reflection groups via extension of scalars.
In this way, an irreducible \emph{complex} reflection group \(G\leq \GL(V')\) acting on a complex vector space \(V'\) gives rise to an irreducible quaternionic reflection group acting on \(V'\otimes_\BBC\BBH\).
Hence the irreducible complex reflection groups classified in \cite{ST54} form a subset of the irreducible quaternionic reflection groups.
The action of \(G\) on \((V'\otimes_\BBC\BBH)|_\BBC\) is (complex) reducible and we consequently call a complex reflection group considered as a quaternionic group a \emph{complex reducible} quaternionic reflection group, see  \Cref{rem:complexify}.

Let \(G\leq \GL(V)\) be an irreducible quaternionic reflection group which is also complex irreducible, so \(G\) is not coming from a complex reflection group.
We call the group \(G\) \emph{imprimitive} if there is a decomposition \(V = V_1\oplus \cdots\oplus V_k\), \(k \geq 2\), into non-trivial spaces \(V_i\) such that the action of every \(g\in G\) on \(V\) permutes the summands \(V_i\).
If no such decomposition exists, then \(G\) is called \emph{primitive}.
The imprimitive irreducible quaternionic reflection groups come in several infinite families in arbitrary dimension \(\dim V \geq 2\); we give a precise description of these groups in \Cref{sec:imprimclass}.
The primitive irreducible quaternionic reflection groups can be divided into infinite families of groups in dimension \(\dim V = 2\) and 13 ``exceptional'' groups in dimension \(2\leq \dim V\leq 5\).
In this article, we focus on the groups with \(\dim V > 2\) because the case \(\dim V = 2\) is trivially handled by \Cref{prop:dim2}.
This means we are mainly concerned with the imprimitive groups and only need to consider the exceptional primitive groups in dimension \(\dim V > 2\), of which there are seven.

While the quaternionic point of view is a natural generalization of complex reflection groups, it is often helpful to turn the quaternionic vector space \(V\) into a complex representation of \(G\) by restriction of scalars.
For the reader's convenience, we give the details of this ``complexification'' construction following Cohen \cite{cohen:quaternionic}.
Consider \(\BBH\) as a right \(\BBC\)-module and choose an  \(\BBH\)-basis \(e_1,\dots,e_n\) of \(V \cong \BBH^n\).
We may write any vector \(v\in V\) as \(v = \sum_{l = 1}^n(x_l + y_l\mathbf j)e_l\) with \(x_l, y_l\in \BBC\) and map $v$ to \[v^\vee := \sum_{l = 1}^nx_l\epsilon_l + \sum_{l = 1}^n\overline{y}_l\epsilon_{l + n}\in V|_{\BBC} \cong \BBC^{2n},\] where \(\epsilon_1,\dots,\epsilon_{2n}\) denotes the standard basis of \(\BBC^{2n}\) and \(\bar{\cdot}\) denotes complex conjugation.
Notice that the complex conjugation in the second component is necessary, if we consider \(\BBH\) as a right \(\BBC\)-module since now we have \[(v\alpha)^\vee = \Big(\sum_{l = 1}^n(x_l\alpha + y_l\overline{\alpha}\mathbf j)e_l\Big)^{\!\!\vee} = \sum_{l = 1}^nx_l\alpha\epsilon_l + \sum_{l = 1}^n\overline{y}_l\alpha\epsilon_{l + n} = v^\vee\alpha\] for every \(\alpha\in \BBC\) as desired.
Similarly, for every matrix \(g \in \GL(V)\), we can write \(g = g_1 + g_2\mathbf j\) with \(g_1,g_2\in \BBC^{n\times n}\) and we map \(g\) to \[g^\vee := \begin{pmatrix} g_1 & -g_2 \\ \overline{g}_2 & \overline{g}_1\end{pmatrix}\in\GL_{2n}(\BBC).\]
Let \(v\in V\) with \(v = v_1 + v_2 \mathbf j\), \(v_1, v_2\in\BBC^n\).
By abuse of notation, we write \(v^\vee = \left(\begin{smallmatrix} v_1\\\overline v_2\end{smallmatrix}\right)\).
  The action of \(\GL(V)\) on \(V\)  is \(^\vee\)-equivariant in the following sense: 
  \[(gv)^\vee = (g_1v_1 - g_2\overline{v}_2 + g_1v_2\mathbf j + g_2\overline{v}_1 \mathbf j)^\vee = \begin{pmatrix} g_1v_1 - g_2\overline v_2 \\ \overline g_1 \overline v_2 + \overline g_2 v_1\end{pmatrix} = g^\vee v^\vee.\]

Without loss of generality, we may assume that \(V = \BBH^n\) and that \(G\) preserves the standard unitary inner product \(\langle \cdot, \cdot\rangle\) on \(V\).
Then \(G^\vee\leq\Sp_{2n}(\BBC)\) preserves the standard symplectic form on \(\BBC^{2n}\), see \cite{cohen:quaternionic} for details.
A complexified quaternionic reflection group is then also called a \emph{symplectic reflection group}. This symplectic point of view is not relevant in the sequel.

\begin{remark}
  \label{rem:complexify}
  Let \(V\) be a \emph{complex} vector space and let \(G\leq\GL(V)\) be a complex reflection group.
  Then we obtain the complexified quaternionic reflection group \[G^\circledast = \left\{\!\begin{pmatrix} g & 0 \\ 0 & \overline{g} \end{pmatrix} \middle |\,g\in G\right\}\leq\GL((V\otimes_\BBC\BBH)|_\BBC).\]
  If we identify \(V = \BBC^n\) and \(G\) preserves the standard unitary inner product on \(V\), then \(\overline g = (g^\top)^{-1}\) and we can consider \(G^\circledast\) as a subgroup of \(\GL(V\oplus V^*)\), where \(V^*\) denotes the dual space of \(V\).
  The isomorphism \((V\otimes_\BBC\BBH)|_\BBC\cong V\oplus V^*\) is given by \[v^\vee = (v_1 + v_2\mathbf j)^\vee \mapsto (v_1, w \mapsto \langle v_2, w\rangle),\] where \(\langle\cdot,\cdot\rangle\) is the standard unitary inner product defined by \(\langle u, w\rangle = \sum_{l = 1}^n\overline u_lw_l\) for vectors \(u = (u_l)_l, w = (w_l)_l\in \BBC^n\).
  We emphasize that \(G\) and \(G^\circledast\) are isomorphic as abstract groups, but we should see them as pairs \((G, V)\) and \((G^\circledast, V\oplus V^*)\) with fixed non-isomorphic complex representations.
\end{remark}

\subsection{Arrangements and their cohomology}
An \emph{arrangement of subspaces} is a pair $(\CA,V)$, where $V$ is a finite-dimensional vector space and $\CA$ is a finite set of linear subspaces of $V$. We omit the ambient space $V$, when it is not relevant, and we call the arrangement real, complex, or quaternionic, if $V$ is a (right) vector space over $\BBR$, $\BBC$, or $\BBH$, respectively. The main combinatorial object associated to $\CA$ is its \emph{lattice of intersections}, that is, the set of all intersections of the subspaces of $\CA$ ordered by reverse inclusion which we denote $L(\CA)$. 
The poset $L(\CA)$ is ranked: 
for \(X\in L(\CA)\), we denote the quaternionic codimension of \(X\) by \(\rk(X)\).
Further, we define 
$\rk(\CA) := \rk(\Cent(\CA))$, where $\Cent(\CA)$ is 
the intersection of all elements of $\CA$.

One of the main geometric objects associated to $(\CA,V)$ is the \emph{complement space} $M(\CA) := V\setminus\cup_{S\in\CA}S$. A recurrent theme in the theory of arrangements is to describe geometric and algebraic properties of $M(\CA)$ in terms of combinatorial properties of $L(\CA)$. When the subspaces are all hyperplanes, $\CA$ is called a \emph{hyperplane arrangement}, its lattice of intersections is a geometric lattice and the integer cohomology ring of $M(\CA)$ is described in terms of the associated matroid by the \emph{Orlik--Solomon algebra}.

For more information on hyperplane arrangements and the Orlik--Solomon algebra we refer to \cite{orlikterao:arrangements}. When the subspaces are not hyperplanes, less is known. For instance, the ring structure of the cohomology was presented uniformly only for certain classes of arrangements. For us, knowing the rational cohomology groups will be sufficient, as we are interested in the cohomology as a $\BBQ G$-module, for some group $G$ acting on $V$ and fixing $\CA$. A nice description of the integer cohomology groups was given by Goresky and MacPherson:

\begin{theorem}[\cite{goreskymacpherson:morse}]
	\label{thm:gm}
  Let $\CA$ be a real subspace arrangement. Then its reduced cohomology groups are described by the formula $$\tilde{H}^k(M(\CA);\mathbb{Z}) \cong \bigoplus_{X\in L(\CA)\setminus \{\hat0\}}\tilde{H}_{\rk(X)-2-k}((\hat0,X);\mathbb{Z}),$$ where the homology on the right hand side refers to the order complex of the interval $(\hat0,X)$ in $L(\CA)$.
\end{theorem}

This result has many useful consequences, we list some that we need.

\begin{defn}
  Let $c$ be a positive integer. A subspace arrangement $\CA$ is called a $c$\emph{-arrangement} provided
  \begin{itemize}
    \item $\codim(S) = c$ for all subspaces $S\in\CA$,
    \item $c$ divides $\codim(X)$ for all $X\in L(\CA)$.
  \end{itemize}
  The lattice of intersections of a $c$-arrangement is a geometric lattice with rank function either $\codim: L(\CA)\rightarrow\mathbb{Z}_{\geq 0}$ or $\rk := \frac{1}{c}\codim$. The latter can be thought of as the rank function of the underlying abstract lattice; it is used in lattice homology. We write \(L(\CA)_k\) for the members of \(L(\CA)\) of rank \(k\).
\end{defn}

\begin{example}
  Let $\CA$ be a complex hyperplane arrangement. Then $\CA$ is a real $2$-arrange\-ment by restriction of scalars.
  Likewise, a quaternionic hyperplane arrangement \(\CB\) can be considered as a complex 2-arrangement or a real 4-arrangement.
  We use the notation \(\CB|_\BBC\) or \(\CB|_\BBR\) to indicate that we consider \(\CB\) as a complex or real subspace arrangement.
  Clearly, not all real $2$ and $4$-arrangements come from complex or quaternionic hyperplane arrangements.
\end{example}

The following is a classical result on the homology of geometric lattices:

\begin{lemma}
	\label{lem:order}
	Let $L$ be a geometric lattice, and $X\in L\setminus \{\hat0\}$. Then the order complex associated to the open interval $(\hat0, X)$ has the homotopy type of a wedge of $|\mu(X)|$ spheres of dimension $\rk(X)-2$. In particular, $$\tilde{H}_k((\hat0,X);\mathbb{Z})\cong\begin{cases} \mathbb{Z}^{|\mu(X)|} &\text{ if } k = \rk(X)-2, \\ 0 &\text{ otherwise.} \end{cases}$$
\end{lemma}

Part \labelcref{cBrieskorn:1} of the following result is an analogue of Brieskorn's Lemma \cite[Lem.~5.91]{orlikterao:arrangements} for real $c$-arrangements.

\begin{prop}
  \label{cBrieskorn}
  Let $\CA$ be a real $c$-arrangement, for $c\geq 2$. 
  \begin{enumerate}[(1)]
    \item\label{cBrieskorn:1}
      For $X\in L(\CA)_k$, the inclusions $M(\CA) \subseteq M(\CA_X)$ induce  isomorphisms
      $$\tilde{H}^k(M(\CA);\mathbb{Z})\cong\bigoplus_{X\in L(\CA)_n} \tilde{H}^{k}(M(\CA_X);\mathbb{Z}),$$ where $n = \frac{k}{c-1}\in\mathbb{Z}$. 
    \item\label{cBrieskorn:2} $\tilde{H}^k(M(\CA);\mathbb{Z})\neq 0$ only if $c-1$ divides $k$.
  \end{enumerate}
\end{prop}

\begin{proof}
  The second statement is an immediate consequence of the first. The first statement is a direct consequence of Goresky--MacPherson's isomorphism from \Cref{thm:gm} and the properties of the lattices of intersections of $c$-arrangements. By \Cref{lem:order}, the interval $(\hat0,X)$ has nontrivial reduced homology only in degree $\rk(X)-2$. Thus, a nontrivial contribution in the right hand side of the isomorphism appears only when $$\codim(X)-2-k = \rk(X)-2,$$ that is, since $\rk = \frac{1}{c}\codim$, for $$\rk(X) = \frac{k}{c-1}.$$ Thus, the Goresky--MacPherson formula becomes 
  $$
  \tilde{H}^k(M(\CA);\mathbb{Z})\cong\bigoplus_{X\in L(\CA)_n} \mathbb{Z}^{|\mu(X)|}\cong\bigoplus_{X\in L(\CA)_n}\tilde{H}^{k}(M(\CA_X);\mathbb{Z}),
  $$
  where $n = \frac{k}{c-1}\in\mathbb{Z}$. 
  The second isomorphism and hence the claim follows from the isomorphism $\mathbb{Z}^{|\mu(X)|} \cong H^{k}(M(\CA_X))$ which again is the Goresky--MacPherson formula in top degree.
\end{proof}

\subsection{Cohomology of the complement for quaternionic hyperplane arrangements}
Let \(\CA\) be a quaternionic hyperplane arrangement in \(\BBH^n\).
We always consider the complement of \(\CA\) in \(\BBH^n\) as a complex space, that is, we study the complement \(M(\CA|_\BBC)\) of the 2-arrangement \(\CA|_\BBC\) inside \(\BBC^{2n}\).
By abuse of notation, we write \(M(\CA) = M(\CA|_\BBC)\).

The cohomology ring \(H^*(M(\CA);\BBZ)\) has an Orlik--Solomon-like presentation, by \cite[Prop.~4]{Sch15}, \cite[Prop.~7]{CS18}.
Because, to our knowledge,  this result is not available within a peer-reviewed publication, we give the following independent proof.
Our argument is a direct consequence of the more general result in \cite[Cor.~5.6]{LS01}.
For two elements \(h, h'\in H^*(M(\CA);\BBZ)\), we frequently write \(hh'\) for the product \(h\wedge h'\).
\begin{prop}
  \label{cohomologyquaternionicarrangement}
  Let \(\CA = \{H_1,\dots,H_m\}\) be a quaternionic hyperplane arrangement.
  The integral cohomology ring of the complex complement \(M(\CA) = M(\CA|_\BBC)\) has the presentation \[\begin{tikzcd} 0 \arrow{r} & I \arrow{r} & \Lambda(\BBZ^m) \arrow["\pi"]{r} & H^*(M(\CA);\BBZ) \arrow{r} & 0\end{tikzcd}\] with \(\pi(e_i) \in H^3(M(\CA);\BBZ)\) for the canonical basis \(\{e_1,\dots, e_m\}\) of \(\BBZ^m\).
  The ideal \(I\) of relations is generated by \[\sum_{i = 0}^k(-1)^ie_{a_0}\wedge\cdots\wedge\widehat{e_{a_i}}\wedge\cdots\wedge e_{a_k},\] for all minimal dependent sets \(\{H_{a_0},\dots, H_{a_k}\}\subseteq \CA\).
\end{prop}
\begin{proof}
  We consider \(\CA\) as a real 4-arrangement by restriction of scalars and use \cite[Cor.~5.6]{LS01} to obtain the desired presentation.
  However, the relations given in \cite{LS01} are \[\sum_{i = 0}^k(-1)^i\epsilon(a_0,\dots,\widehat{a_i},\dots,a_k)e_{a_0}\wedge\cdots\wedge\widehat{e_{a_i}}\wedge\cdots\wedge e_{a_k}\] with additional signs \(\epsilon(a_0,\dots,\widehat{a_i},\dots,a_k)\in\{\pm1\}\) and it remains to prove that we in fact have \(\epsilon(a_0,\dots,\widehat{a_i},\dots,a_k) = +1\) for 4-arrangements coming from a quaternionic arrangement.

  For this, let \(\{H_{a_0},\dots,H_{a_k}\}\subseteq \CA\) be a minimal dependent set.
  To avoid a cluttered notation, we write \(H_{a_j}\) for both the quaternionic hyperplanes and the real subspaces \(H_{a_j}|_\BBR\).
  Let \(V\) be the real vector space associated to the 4-arrangement \(\CA|_\BBR\).
  By \cite[Rem.~5.7]{LS01}, the sign \(\epsilon(a_0,\dots,\widehat{a_i},\dots,a_k)\) is given by the degree of the linear isomorphism \[\pi_i:V/(H_{a_0}\cap\cdots\cap\widehat{H_{a_i}}\cap\cdots\cap H_{a_k}) \to V/H_{a_0}\times\cdots\times\widehat{V/H_{a_i}}\times\cdots\times V/H_{a_k}.\]
  Our computation of this determinant is now similar to the argument in \cite[Thm.~4.1]{Zie93}.

  Every quaternionic hyperplane \(H_{a_j}\) is given by a linear form \(f_j:V\otimes_\BBR\BBH\to \BBH\), which we can decompose as \(f_j = f_j^{(1)} + f_j^{(2)}\mathbf i + f_j^{(3)}\mathbf j + f_j^{(4)}\mathbf k\) with real linear forms \(f_j^{(l)}:V\to \BBR\).
  As \(\{H_{a_0},\dots,H_{a_k}\}\) is a dependent set, there are elements \(0\neq \alpha_j\in \BBH\) with \(\sum_{j = 0}^k\alpha_jf_j = 0\).
  Restricting to \(\BBR\) again, we have \(\alpha_j = \alpha_j^{(1)} + \alpha_j^{(2)}\mathbf i + \alpha_j^{(3)}\mathbf j + \alpha_j^{(4)}\mathbf k\) with \(\alpha_j^{(l)}\in \BBR\).
  Consider the matrices
  \[A_j := \begin{pmatrix}
    \alpha_j^{(1)} & \alpha_j^{(2)} & \alpha_j^{(3)} & \alpha_j^{(4)}\\
    -\alpha_j^{(2)} & \alpha_j^{(1)} & \alpha_j^{(4)} &-\alpha_j^{(3)}\\
    -\alpha_j^{(3)} &-\alpha_j^{(4)} & \alpha_j^{(1)} & \alpha_j^{(2)}\\
    -\alpha_j^{(4)} & \alpha_j^{(3)} &-\alpha_j^{(2)} & \alpha_j^{(1)}
  \end{pmatrix}\]
  and let \[\begin{pmatrix} x_j^{(1)} \\ x_j^{(2)} \\ x_j^{(3)} \\ x_j^{(4)} \end{pmatrix} := A_j\begin{pmatrix} f_j^{(1)} \\ f_j^{(2)} \\ f_j^{(3)} \\ f_j^{(4)} \end{pmatrix}\] with \(0\leq j\leq k\).
    We obtain the four real dependencies \(\sum_{j = 0}^k x_j^{(l)} = 0\) for \(l \in \{1, \dots, 4\}\).
  (Notice that elements of \(\BBH\) act conjugated on the dual space \((V\otimes_\BBR\BBH)^*\).)

  The families \(\{f_j^{(1)},\dots,f_j^{(4)}\}\) give bases of the quotient spaces \(V/H_{a_j}\) via the isomorphism \(V\cong V^*\).
  The family \(\{x_i^{(1)},x_i^{(2)}, x_i^{(3)}, x_i^{(4)}\}\) is linearly independent for every \(0\leq i\leq k\) because \(H_{a_i}\) is of codimension 4 in \(V\).
  As \(\{H_{a_0},\dots, H_{a_k}\}\) is a \emph{minimal} dependent set, we therefore have a linearly independent family  \(B_i := \{x_j^{(l)}\mid j\neq i,\ 1\leq l\leq 4\}\) for every \(i\).
  Hence, \(B_i\) gives a basis for the quotient space \(V/(H_{a_0}\cap\cdots\cap\widehat{H_{a_i}}\cap\cdots\cap H_{a_k})\).
  Write \(d_i := \det A_i\).
  Then the determinant of the linear map \(\pi_i\) is given by \(d_0^{-1}\cdots\widehat{d_i^{-1}}\cdots d_k^{-1}\).
  A direct computation gives
  \begin{align*}
    d_i &= (\alpha_i^{(1)})^4 + 2(\alpha_i^{(1)})^2(\alpha_i^{(2)})^2 + 2(\alpha_i^{(1)})^2(\alpha_i^{(3)})^2 + 2(\alpha_i^{(1)})^2(\alpha_i^{(4)})^2\\
    &+ (\alpha_i^{(2)})^4 + 2(\alpha_i^{(2)})^2(\alpha_i^{(3)})^2 + 2(\alpha_i^{(2)})^2(\alpha_i^{(4)})^2\\
    &+ (\alpha_i^{(3)})^4 + 2(\alpha_i^{(3)})^2(\alpha_i^{(4)})^2 + (\alpha_i^{(4)})^4,
  \end{align*}
  so \(d_i > 0\) and we conclude \(\epsilon(a_0,\dots,\widehat{a_i},\dots,a_k) = +1\) for all \(i\).
\end{proof}

Thus, the integer cohomology of complements of quaternionic arrangements is still isomorphic to an Orlik--Solomon algebra,  only with generators in degree $3$. Brieskorn's Lemma for quaternionic arrngements can then be deduced directly from \Cref{cohomologyquaternionicarrangement} together with \cite[Lem.~5.91]{orlikterao:arrangements}. Finally, we give explicit generators for the cohomology of quaternionic arrangements.

\begin{defn}
	Let $(\CA,V)$ be a quaternionic arrangement, and $H\in\CA$ with $H = \ker(\alpha_H)$. By abuse of notation, denote with $\alpha_H:V\setminus H\to \BBH^\times$ the restricted map $\alpha_H|_{V\setminus H}$. The map $\alpha_H$ and the inclusion $\iota_H: M(\CA)\hookrightarrow V\setminus H$ induce maps in cohomology:
	\begin{align*}
		& \alpha_H^*:H^*(\mathbb{H}^\times;\mathbb{Z})\rightarrow H^*(V\setminus H;\mathbb{Z}), \text{ and }\\
		& \iota_H^* : H^*(V\setminus H;\mathbb{Z}) \rightarrow H^*(M(\CA);\mathbb{Z}).
	\end{align*}
	Notice that $\mathbb{H}^\times$ is homeomorphic to $\mathbb{C}^2\setminus\{0\}$ and homotopically equivalent to $S^3$. Let $\omega$ be a generator of $H^3(\mathbb{H}^\times;\mathbb{Z})$. Similarly to what is done for complex arrangements in \cite{orlikterao:arrangements}, define $$e_H := \iota_H^*\alpha_H^*(\omega)\in H^3(M(\CA);\mathbb{Z})$$ as the generator corresponding to the hyperplane \(H\in\CA\).
\end{defn}

\section{Quaternionic reflection arrangements}
From this section onward, we consider cohomology of complement spaces with rational coefficients: we see from \Cref{cohomologyquaternionicarrangement} that there is no torsion with integer coefficients, so by the universal coefficient theorem, we have $H^*(M(\CA);\mathbb{Q}) \cong H^*(M(\CA);\mathbb{Z})$. If we don't make coefficients explicit, we always work in $\mathbb{Q}$. 

Throughout, let \(V\) be a finite-dimensional right \(\BBH\)-vector space and let \(G\leq \GL(V)\) be a quaternionic reflection group.

\subsection{Reflection arrangements}
\label{sec:reflarr}
In the following, we denote by \(\Fix(g)\) the pointwise fixed space of an element \(g\in \GL(V)\).

\begin{defn}
	\label{def:quatreflarr}
  Let \(G\leq \GL(V)\) be a quaternionic reflection group.
  We call the set \[\CA(G) := \{\Fix(g) \mid g\in G\text{ quaternionic reflection}\}\] the \emph{(quaternionic) reflection arrangement} of \(G\).
\end{defn}

We have the following direct analogue for quaternionic reflection arrangements of the well-known result \cite[Thm.~6.27]{orlikterao:arrangements}.
\begin{theorem}
  \label{thm:paratoint}
  Let \(G\leq \GL(V)\) be a quaternionic reflection group.
  \begin{enumerate}[(1)]
    \item\label{thm:paratoint:1} If \(g\in G\), then \(\Fix(g)\in L(\CA(G))\).
    \item If \(X\in L(\CA(G))\), then there exists \(g\in G\) with \(\Fix(g) = X\).
  \end{enumerate}
\end{theorem}
\begin{proof}
  The first claim follows as in \cite[Thm.~6.27]{orlikterao:arrangements} using \cite{BST23}.
  For the second claim, one may use the argument from \cite{orlikterao:arrangements} with a minor modification.
  Namely, the stabilizer \(G_H\) of a hyperplane \(H\in\CA(G)\) is not necessarily cyclic.
  However, one may choose the elements \(s_i\) in \cite[Thm.~6.27]{orlikterao:arrangements} to be some elements of \(G\) with the corresponding fixed spaces; the condition that they generate the stabilizers is not required.
\end{proof}

Let \(\CA = \CA(G)\) be the reflection arrangement of \(G\).
The action of \(G\) on \(V\) induces an action of \(G\) on \(\CA\).
Precisely, for $H_r = \Fix(r)\in \CA(G)$, the hyperplane corresponding to a reflection \(r\in G\), we have \(g.H_r = H_{grg^{-1}}\) for \(g\in G\).
This action extends to an action of \(G\) on the lattice of intersections \(L(\CA)\).

We call a subgroup \(P\leq G\) a \emph{parabolic subgroup} of \(G\) if \(P\) is the pointwise stabilizer in \(G\) of a subset of \(V\).
By \cite{BST23}, a parabolic subgroup is again a quaternionic reflection group.
The set of parabolic subgroups of \(G\) is partially ordered by inclusion and hence forms a poset, which we denote by \(\mathcal P(G)\).
The group \(G\) acts on \(\mathcal P(G)\) by conjugation.

We have the following consequence of \Cref{thm:paratoint}.
\begin{corollary}
  \label{cor:paratoint}
  Let \(G\) be a quaternionic reflection group and let \(\CA\) be the reflection arrangement of \(G\).
  There is a \(G\)-equivariant isomorphism of lattices \(L(\CA) \cong \mathcal P(G)\).
\end{corollary}
\begin{proof}
  By \Cref{thm:paratoint}, the fixed space of any parabolic subgroup of \(G\) is an element of \(L(\CA)\).
  On the other hand, for any \(X\in L(\CA)\), there is \(g\in G\) with \(\Fix(g) = X\) by the theorem again.
  So, taking the pointwise stabilizer of \(X\) in \(G\) gives a parabolic subgroup \(P_X\) with fixed space \(X\).
  All in all, we have an order preserving bijection between the poset \(\mathcal P(G)\) and the lattice \(L(\CA)\).
  This isomorphism is \(G\)-equivariant because \(\Fix(gPg^{-1}) = g\Fix(P)\) for every \(g\in G\) and every parabolic subgroup \(P\leq G\).
\end{proof}

\subsection{Invariants}
\label{sec:invariants}

The action of \(G\) on \(\CA\) induces an action of \(G\) on the cohomology spaces \(H^*(M(\CA);\BBQ)\), as follows.
Here and in the following, we simply write \(hh'\) for the product \(h\wedge h'\) of elements \(h, h'\in H^*(M(\CA), \BBQ)\).
By \Cref{cohomologyquaternionicarrangement}, we have that $H^3(M(\CA);\BBQ)$ is the $\BBQ$-vector space with basis $e_H$ for $H\in\CA$.
Hence $G$ acts on \(H^*(M(\CA);\BBQ)\) by permuting the generators \(e_H\) via 
\begin{align*}g.e_H &= e_{g.H}, \\ g.(e_{H_1}\cdots e_{H_k}) &= (g.e_{H_1})\cdots (g.e_{H_k}),
\end{align*} 
for \(H, H_1,\dots, H_k\in \CA\).
In particular, $H^3(M(\CA);\BBQ)$ affords the permutation representation of $G$  on \(\CA\), and the cohomology is endowed with a $\BBQ G$-module structure.

Let $\mathcal{X}(\CA,G)$ be a set of representatives of the orbits of the action of $G$ on $L(\CA)$.
Notice that the action of \(G\) on \(L(\CA)\) maintains the rank and let $\mathcal{X}(\CA,G)_k = \mathcal{X}(\CA,G)\cap L(\CA)_k$, where 
\(L(\CA)_k\) denotes the members of \(L(\CA)\) of rank \(k\), i.e.~of quaternionic codimension \(k\).

The following result is an analogue of \cite[Prop.~2.5]{douglaspfeifferroehrle:invariants} for quaternionic reflection arrangements adapted to our purposes.
We write 
\[\epsilon_G := \frac{1}{|{G}|}\sum_{g\in G}g\in \BBQ G\]
 for the primitive idempotent in \(\BBQ G\).

\begin{prop}
  \label{Gdecomp}
  Let \(n = \dim V\) and $k \in\{0,\ldots,n\}$.
  \begin{enumerate}[(1)]
    \item\label{Gdecomp:1} For $X\in L(\CA)_k$, the inclusion $M(\CA) \subseteq M(\CA_X)$ induces isomorphisms of $\BBQ G$-modules:
      \begin{equation*}
        H^{3k}(M(\CA);\BBQ) \simeq \bigoplus_{X\in L(\CA)_k} H^{3k}(M(\CA_X);\BBQ) \simeq \bigoplus_{X\in \mathcal{X}(\CA,G)_k}\Ind_{N_G(X)}^{G}\left( H^{3k}(M(\CA_X);\BBQ) \right).
      \end{equation*}
    \item\label{Gdecomp:2} For \(X\in L(\CA)\), multiplication by \(\epsilon_G\) gives an isomorphism \[\epsilon_G H^{3k}(M(\CA_X);\BBQ)^{N_G(X)}\cong \Big(\bigoplus_{Y\in G.X} H^{3k}(M(\CA_Y);\BBQ)\!\Big)^{\!G},\] where \(G.X\) denotes the \(G\)-orbit of \(X\) in \(L(\CA)\).
      Summing over \(X\in \mathcal X(\CA, G)_k\), the first isomorphism in \labelcref{Gdecomp:1} gives the equality \[H^{3k}(M(\CA);\BBQ)^G = \sum_{X\in \mathcal X(\CA, G)_k}\epsilon_G\cdot H^{3k}(M(\CA_X);\BBQ)^{N_G(X)}.\]
    \item\label{Gdecomp:3} We have an isomorphism \[H^{3k}(M(\CA);\BBQ)^G \cong \bigoplus_{X\in\mathcal{X}(\CA,G)_k}H^{3k}(M(\CA(Z_G(X)));\BBQ)^{N_G(X)}.\]
  \end{enumerate}
\end{prop}

\begin{proof}
  The proof is analogous to the proof of \cite[Prop.~2.5]{douglaspfeifferroehrle:invariants}. The first isomorphism in \labelcref{Gdecomp:1} is the content of \Cref{cBrieskorn}. The second one follows as in \emph{loc.~cit.}\ from work of Lehrer and Solomon \cite{LS86} and Orlik and Solomon \cite{OS83}. In the latter, the results are stated in terms of the poset of parabolic subgroups of a reflection group, thus we conclude by \Cref{cor:paratoint}.

  The statement in \labelcref{Gdecomp:2} is a claim about \(\BBQ G\)-modules and follows as in \cite{douglaspfeifferroehrle:invariants}.

  In \labelcref{Gdecomp:3}, the group \(Z_G(X)\) is a quaternionic reflection group by \cite{BST23} and we have \(\CA_X = \CA(Z_G(X))\) for \(X\in L(\CA)\).
  Hence the isomorphism in \labelcref{Gdecomp:3} follows from \labelcref{Gdecomp:1} and \labelcref{Gdecomp:2}.
\end{proof}

\begin{remark}
  \label{rem:transfer}
  Topologically, we may interpret the invariants \(H^*(M(\CA);\BBQ)^G\) as follows.
  Let \(M(\CA)/G\) be the quotient space of \(M(\CA)\) by \(G\).
  By \Cref{thm:paratoint} \labelcref{thm:paratoint:1}, the action of \(G\) on \(M(\CA)\) is free.
  As \(G\) is finite and \(M(\CA)\) a Hausdorff space, we conclude that the projection \(\rho: M(\CA) \to M(\CA)/G\) is a normal covering space by \cite[Prop.~1.40]{Hat02}.
  Then there is an isomorphism \[H^*(M(\CA);\BBQ)^G \cong H^*(M(\CA)/G; \BBQ)\] by the `transfer homomorphism' from algebraic topology \cite[Prop.~3G.1]{Hat02}.
\end{remark}

\begin{remark}
  We record a further consequence of \(\rho:M(\CA) \to M(\CA)/G\) being a covering space, see Remark~\ref{rem:transfer}.
  By \cite[Prop.~1.40]{Hat02}, we have an isomorphism \[G \cong \pi_1(M(\CA)/G)/\rho_*(\pi_1(M(\CA))).\]
  Because a quaternionic hyperplane has real codimension 4, the space \(M(\CA)\) is simply connected.
  We conclude \(\pi_1(M(\CA)) = \{0\}\) and \(\pi_1(M(\CA)/G) \cong G\).
  Notice that in the setting of real or complex reflection groups, the spaces \(M(\CA)\) and \(M(\CA)/G\)  are $K(\pi,1)$ \cite{bessis:finite},  and their fundamental groups \(\pi_1(M(\CA))\) and \(\pi_1(M(\CA)/G)\) are called the \emph{pure braid group} \(PB_G\) and the \emph{braid group} \(B_G\), respectively: see \cite{BMR98}.  In particular,  the group cohomology of $B_G$ is isomorphic to the cohomology of \(M(\CA)/G\), and hence \[H^*(B_G;\mathbb{Q})\cong H^*(M(\CA);\mathbb{Q})^G.\]
  However, in the quaternionic case this terminology does not seem to apply and \(\pi_1(M(\CA)/G)\) does not contain any new structure.
  Finally,  one should note that there are infinitely many \(k \geq 2\) for which the homotopy group \(\pi_k(M(\CA))\) is non-trivial, see \cite[Cor.~3.2]{BW95}, and we have \(\pi_k(M(\CA)) \cong \pi_k(M(\CA)/G)\) for \(k \geq 2\) by \cite[Prop.~4.1]{Hat02}. 
\end{remark}

Due to the Orlik--Solomon presentation of $H^*(M(\CA);\BBQ)$ in \Cref{cohomologyquaternionicarrangement}, the main statements in Section~2 of \cite{douglaspfeifferroehrle:invariants} hold as well in our setting.  For convenience of the reader, we restate the relevant results. Their geometric counterparts in the quaternionic setting are analyzed in more detail in \cite{Sch15}.  From \cite[Def.~3.12]{orlikterao:arrangements}, we have a homogeneous derivation $\partial: H^*(M(\CA);\BBQ)\rightarrow H^*(M(\CA);\BBQ)$ of degree $-3$ mapping $e_H\in H^{3}(M(\CA);\BBQ)$ to $1$ defined by $$\partial(e_{H_1}\cdots e_{H_k}) = \sum_{i=1}^{k}(-1)^{i - 1}e_{H_1}\cdots \widehat{e_{H_i}}\cdots e_{H_k}.$$ By \cite[Lem.~3.13]{orlikterao:arrangements}, $(H^*(M(\CA);\BBQ),\partial)$ is an acyclic complex, and by the remarks at the beginning of the section, $\partial$ is $G$-equivariant. Following \cite[\S 2.7]{douglaspfeifferroehrle:invariants}, define the map $\mu: H^*(M(\CA);\BBQ)\rightarrow H^*(M(\CA);\BBQ)$ by $$\mu(x) = \left(\frac{1}{|\CA|}\sum_{H\in\CA}e_H\!\right)\!x .$$ One sees immediately that $\mu$ is homogeneous of degree $3$, $G$-equivariant and satisfies the relation $\mu\partial + \partial\mu = \id$. For $k=0,\ldots,\rk(\CA)$ define 
$$
\overline{H^{3k}} := \partial\!\left(H^{3(k+1)}(M(\CA);\BBQ)\right).
$$ 
With the convention $H^{-3}(M(\CA);\BBQ) = \{0\}$, there is a canonical direct sum decomposition 
$$ 
H^{3k}(M(\CA);\BBQ) \simeq \mu(\overline{H^{3(k-1)}})\oplus \overline{H^{3k}}.
$$ 
As a consequence, we derive the following crucial Euler characteristic-like identity: 
$$
\sum_{k=0}^{\rk(\CA)} (-1)^{k}\dim H^{3k}(M(\CA);\BBQ)^G = 0.
$$

Together with \Cref{Gdecomp}, we finally have 
\begin{align}
  \label{eq:euler}
  \sum_{X\in\mathcal{X}\setminus \{\Cent(\CA)\}}\!\!\big((-1)^{\rk(X)}\dim H^{3\rk(X)}(M(\CA_X);\BBQ)^G \big)+ (-1)^n\dim H^{3n}(M(\CA);\BBQ)^G = 0,
\end{align} 
where \(\rk(X)\) for \(X\in L(\CA)\) denotes the quaternionic codimension of \(X\), as above, and \(n = \rk(\CA)\).
This is our inductive tool to compute the Poincaré polynomial 
$P(\SA,G; t)$
of $H^*(M(\CA);\BBQ)^G$: indeed, if one has computed $H^{3\rk(X)}(M(\CA_X);\BBQ)^G$ for $X\neq\Cent(\CA)$, then one can immediately recover the case $X = \Cent(\CA)$ with the above formula.
Finally, we have the following result for the Poincaré polynomial of the invariants for $\rk(\CA) = 2$ which follows by the same arguments as in \cite[Prop.~2.9]{douglaspfeifferroehrle:invariants} using \eqref{eq:euler}:

\begin{prop}
  \label{prop:dim2}
  Let \(G\leq\GL(V)\) be a quaternionic reflection group with \(\dim V = 2\).
  \begin{enumerate}
    \item If $G$ acts on $\CA$ with $a$ orbits, then $$ P(\CA(G),G;t) = 1 + at^3 + (a-1)t^6.$$
    \item If $\{H_1,\ldots,H_a\}$ is a set of orbit representatives for the action of $G$ on $\CA$, then the following is a graded basis for $H^{*}(M(\CA);\BBQ)^{G}$: $$ \{1\} \cup \{\epsilon_G\cdot e_{H_1},\ldots ,\epsilon_G\cdot e_{H_a}\} \cup \{\epsilon_G\cdot e_{H_1}e_{H_2},\ldots, \epsilon_G\cdot e_{H_1}e_{H_a}\},$$
      where $\epsilon_G = \frac{1}{|G|}\sum_{g\in G} g$.
  \end{enumerate}
\end{prop}

\begin{remark}
  Let \(G\leq\GL(V)\) be a reducible reflection group leaving the decomposition \(V = V_1\oplus V_2\) invariant and let \(G_i\leq \GL(V_i)\) with \(G = G_1\times G_2\).
  Then the Künneth formula induces an isomorphism \[H^*(M(\CA(G)))^G \cong H^*(M(\CA(G_1)))^{G_1} \otimes_\BBQ H^*(M(\CA(G_2)))^{G_2},\] see also \cite[\S 3]{douglaspfeifferroehrle:invariants} for more details.
  We may hence restrict to irreducible reflection groups in the following.
\end{remark}

\subsection{Invariants in the complex reducible case}

In this section only, let \(V\) be a complex vector space and \(G\leq \GL(V)\) be an irreducible complex reflection group with reflection arrangement \(\CA_\BBC\).
As explained in \Cref{sec:quatreflgroups}, we may consider \(G\) as a quaternionic reflection group acting on \(V\otimes_\BBC\BBH\) and this representation of \(G\) is quaternionic irreducible, but complex reducible.
We write \(\CA_\BBH\) for the quaternionic reflection arrangement of \(G\) as a quaternionic group.
Naturally, there is a \(G\)-equivariant bijection between \(\CA_\BBC\) and \(\CA_\BBH\).

The following is now a direct consequence of \Cref{cohomologyquaternionicarrangement}.
\begin{lemma}
  Let \(G\leq\GL(V)\) be a complex reflection group with reflection arrangement \(\CA_\BBC\) and corresponding quaternionic reflection arrangement \(\CA_\BBH\).
  Then there is a \(G\)-equivariant graded isomorphism of algebras \[H^*(M(\CA_\BBC);\BBQ) \cong H^\ast(M(\CA_\BBH);\BBQ)\] sending a generator \(e_{H_r}\in H^1(M(\CA_\BBC);\BBQ)\) to \(e_{H_r\otimes_\BBC\BBH}\in H^3(M(\CA_\BBH);\BBQ)\), where \(r\in G\) is a reflection.
\end{lemma}

By \cite{douglaspfeifferroehrle:invariants}, there are only four possible types of polynomials arising for $G$ an irreducible complex reflection group.
We repeat the result from \cite{douglaspfeifferroehrle:invariants} for completeness.
In the theorem, we use the labelling of irreducible complex reflection groups from \cite{ST54}.

\begin{theorem}
  \label{thm:poincarecompred}
  Let $G$ be an irreducible complex reflection group acting on $V$ of rank at least $2$.
  Let \(\CA_\BBC\) be the complex reflection arrangement of \(G\) and \(\CA_\BBH\) the corresponding quaternionic reflection arrangement.
  Then the Poincaré polynomial of $H^{\ast}(M(\CA_\BBH);\BBQ)^{G}$ is \[P(\CA_\BBH,G;t) = P(\CA_\BBC,G;t^3).\]
  More precisely, owing to \cite{douglaspfeifferroehrle:invariants}, we have the following cases:
  \begin{enumerate}
    \item for $G$ one of the following groups: $G(r,r,n)$ for $n$ or $r$ odd, $G_4,$ $G_8,$ $G_{12},$ $G_{16},$ $G_{20},$ $G_{22},$ $G_{25},$ $G_{32},$ or $E_6$, we have \[P(\CA_\BBH,G;t) = 1+t^3,\]
    \item for $G$ one of the groups $G(r,r,n)$ for $n$ and $r$ even, $H_{3},$ $G_{24},$ $G_{27},$ $G_{29},$ $H_4,$ $G_{31},$ $G_{33},$ $G_{34},$ $E_7$, or $E_8$, we have \[P(\CA_\BBH,G;t) = 1 + t^3 + t^{3n-3}+t^{3n},\]
    \item for $G$ one of the groups $G(r,p,n)$ with $p<r$ and $n$ or $p$ odd, $G_5,$ $G_6,$ $G_9,$ $G_{10},$ $G_{13},$ $G_{14},$ $G_{17},$ $G_{18},$ $G_{21},$ $G_{26}$, or $F_4$, we have \[P(\CA_\BBH, G;t) =1 + 2t^3 + \cdots + 2t^{3n-3} + t^{3n},\]
    \item for $G$ one of the groups $G(r,p,n)$ with $p<r$ and both $n$ and $p$ even, $G_7,$ $G_{11},$ $G_{15},$ or $G_{19}$, we have \[P(\CA_\BBH,G;t) = 1 + 2t^3 + \cdots + 2t^{3n-6} + 3t^{3n-3} + 2t^{3n}.\]
  \end{enumerate}
\end{theorem}

\section{Imprimitive quaternionic reflection groups}

\subsection{The classification of imprimitive reflection groups}
\label{sec:imprimclass}

Let \(V\) be a finite-dimensional right vector space over \(\BBH\) of dimension \(n\geq 3\) and let \(G\leq \GL(V)\) be a reflection group.
Recall that \(G\) is called \emph{imprimitive} if there is a decomposition \(V = V_1\oplus \cdots\oplus V_k\), \(k \geq 2\), into non-trivial spaces \(V_i\) such that the action of every \(g\in G\) on \(V\) permutes the summands \(V_i\).
By \cite[Thm.~2.9]{cohen:quaternionic}, the irreducible, imprimitive quaternionic reflection groups with \(\dim V \geq 3\) are given by normal subgroups of certain wreath products.
More precisely, let \(K, H\leq \BBH^\times\) be finite groups with \([K,K]\leq H\trianglelefteq K\) and let \[A_n(K, H) = \left\{\left(\begin{smallmatrix}k_1&&\\&\ddots&\\&&k_n\end{smallmatrix}\right)\;\middle|\;k_1,\dots,k_n\in K,\ k_1\cdots k_n\in H\right\}\leq\GL_n(\BBH).\]
Then every irreducible, imprimitive quaternionic reflection group \(G\) acting on a space of dimension \(n\) is conjugate to a group \[G_n(K, H) := A_n(K, H) \rtimes S_n,\] where \(S_n\) acts on an element of \(A_n(K, H)\) by permuting the entries on the diagonal in the natural way.
Note that for $H = K$, we have \(A_n(K, K) \cong K^n\) and so 
\(G_n(K, K) = K\wr S_n\) is a wreath product and in general, \(G_n(K, H)\trianglelefteq G_n(K, K)\) is a normal subgroup.

We have the following list of finite subgroups of \(\BBH^\times\):
\begin{itemize}
  \item the cyclic groups \(\mathsf C_d\), \(d\geq 1\), of order \(d\);
  \item the binary dihedral groups \(\mathsf D_d\), \(d\geq 2\), of order \(4d\);
  \item the binary tetrahedral, binary octahedral and binary icosahedral groups \(\mathsf T\), \(\mathsf O\), \(\mathsf I\) of order 24, 48 and 120, respectively.
\end{itemize}
See \cite[Ex.~1.1]{cohen:quaternionic} for generators of these groups.
In \Cref{tab:kleinian}, we list the groups \(H\) that can occur for a given group \(K\) together with the derived subgroups and the isomorphism type of the abelianization of \(K\), see \cite[Ch.~20]{DuV64} for a classical reference of these results.

\begin{table}
  \caption{Finite subgroups \(K\leq\BBH^\times\) and possible groups \([K, K]\leq H\trianglelefteq K\)}
  \label{tab:kleinian}
  \begin{tabular}{ccll}
    \hline
    \(K\) & \([K, K]\) & \(K/[K, K]\) & \(H\)\\
    \hline
    \(\mathsf C_d\) & \(\mathsf C_1\) & \(C_d\) & \(\mathsf C_e\) for \(e\mid d\)\\
    \(\mathsf D_d\) & \(\mathsf C_d\) & \(C_2\times C_2\) (\(d\) even), & \(\mathsf C_d\), \(\mathsf C_{2d}\), \(\mathsf D_d\), \(\mathsf D_{d/2}\) (\(d\) even)\\
    & & \(C_4\) (\(d\) odd) & \(\mathsf C_d\), \(\mathsf C_{2d}\), \(\mathsf D_d\) (\(d\) odd)\\
    \(\mathsf T\) & \(\mathsf D_2\) & \(C_3\) & \(\mathsf D_2\), \(\mathsf T\)\\
    \(\mathsf O\) & \(\mathsf T\) & \(C_2\) & \(\mathsf T\), \(\mathsf O\)\\
    \(\mathsf I\) & \(\mathsf I\) & \(\{1\}\) & \(\mathsf I\)\\
    \hline
  \end{tabular}
\end{table}

For \(K\) cyclic, the group \(G_n(\mathsf C_d, \mathsf C_e)\) stems from a complex reflection group; we have the equality \(G_n(\mathsf C_d, \mathsf C_e)^\vee = G(d, d/e, n)^\circledast\) with the usual notation from \cite{ST54}, see \Cref{rem:complexify}.

\begin{remark}
  \label{rem:imprimarr}
  Assume \(H\neq \{1\}\).
  One checks that the hyperplanes in the reflection arrangement \(\CA(G_n(K, H))\) are given by \[\ker(x_i),\ 1\leq i\leq n,\text{ and } \ker(x_i - \zeta x_j),\ 1\leq i \neq j \leq n,\ \zeta\in K.\]
  In particular, \(\CA(G_n(K, H))\) only depends on \(K\) and \(n\), but not on \(H\).
  We introduce the notation \(\CA_n(K) := \CA(G_n(K, K))\) for this arrangement.
\end{remark}

\begin{remark}
  \label{rem:noH1}
  If \(H = \{1\}\), then the arrangement \(\CA(G_n(K, H))\) does not contain the coordinate hyperplanes \(\ker(x_i)\).
  Hence this case must be treated separately.
  However, \(H = \{1\}\) can only occur, if \(K\) is cyclic because \([K, K]\leq H\).
  Then \(G_n(K, H)\) can be identified with a complex reflection group and the results in the following sections are well-known and can be found in the cited references.
  For this reason, we usually restrict to the case \(H\neq \{1\}\).
\end{remark}

\subsection{The poset of parabolic subgroups}
Let \(G = G_n(K, H)\) be an imprimitive quaternionic reflection group acting on \(V = \BBH^n\) with \(H\trianglelefteq K\leq \BBH^\times\) finite groups.
In the following, we describe the parabolic subgroups of \(G\) in detail and prove that \(\mathcal P(G)\) is isomorphic to the \emph{Dowling lattice} \(\mathcal D_n(K)\) \cite{Dow73}.
The results in this section are not surprising for a reader familiar with the poset of parabolic subgroups of an imprimitive complex reflection group \(G(m, p, n)\) and our arguments are largely analogous to the complex case.
Still, we are not aware of a reference handling the quaternionic case in the literature.

Write \(I = \{1,\dots, n\}\).
Let \(\Lambda = \{e_1,\dots, e_n\}\) be a basis of \(V\) so that \(G\) has system of imprimitivity \((\langle e_1\rangle, \dots, \langle e_n\rangle)\).
The following is essentially \cite[Def.~3.3]{MT18}.
\begin{defn}
  Let \(I_0\subseteq \{1,\dots, n\}\), let \(\Pi = (I_1,\dots,I_d)\) be a partition of \(I\setminus I_0\) and let \(\xi:I\setminus I_0 \to K\) be a function.
  Let \(n_i := |{I_i}|\) and define the subgroup \(P_{(I_0, \Pi, \xi)}\) of \(G_n(K, H)\) by \[P_{(I_0,\Pi, \xi)} = P_0 \times P_1 \times \cdots \times P_d,\] where \(P_0\) is the quaternionic reflection group \(G_{n_0}(K, H)\) acting on the space spanned by \(\{e_i\mid i\in I_0\}\) and, for \(1\leq i\leq d\), \(P_i\) is the quaternionic reflection group \(S_{n_i}\) permuting the vectors \(\{\xi(j)e_j\mid j\in I_i\}\).
  The factor \(P_0\) is omitted if \(I_0 = \emptyset\).
  If \(K = \{1\}\), we require \(I_0 = \emptyset\).
\end{defn}

The group \(P_{(I_0, \Pi, \xi)}\) is the pointwise stabilizer of the vectors \(\sum_{i\in I_j}\xi(i)e_i\) for \(1\leq j\leq d\), so \(P_{(I_0, \Pi, \xi)}\) is a parabolic subgroup of \(G\).
We refer to the triple \((I_0, \Pi, \xi)\) as a \emph{parabolic triple}.

\begin{prop}
  \label{prop:paraclass}
  Let \(P\leq G_n(K, H)\) be a parabolic subgroup.
  Then there is a parabolic triple \((I_0, \Pi, \xi)\) with \(P = P_{(I_0, \Pi, \xi)}\).
\end{prop}
\begin{proof}
  The proof of \cite[Prop.~3.4]{BST23} shows that \(P\) is \((K\wr S_n)\)-conjugate to the group \(G_{n_0}(K, H) \times S_{n_1}\times \cdots \times S_{n_d}\) with \(n_0 + \cdots + n_d = n\).
  The elements of \(K\wr S_n = G_n(K, K)\) are given by products \(\diag(g_1,\dots, g_n)M(\sigma)\) where \(\diag(g_1,\dots,g_n)\) is the matrix with \(g_1,\dots, g_n\in K\) on the diagonal and \(M(\sigma)\) is the \(n\times n\) permutation matrix corresponding to \(\sigma\in S_n\).
  So there is such an element with \[P = \diag(g_1,\dots,g_n)M(\sigma)(G_{n_0}(K, H) \times S_{n_1}\times \cdots \times S_{n_d})M(\sigma^{-1}) \diag(g_1^{-1},\dots,g_n^{-1}).\]

  Write \(s_j := \sum_{k = 0}^jn_k\) for \(-1\leq j \leq d\) and set \(I_j = \{\sigma(i)\mid s_{j - 1} < i\leq s_j\}\) for \(0\leq j\leq d\).
  Clearly, \(\Pi := (I_1,\dots, I_d)\) is a partition of \(I\setminus I_0\).
  Let \[P_j := \diag(g_{s_{j - 1} + 1}, \dots, g_{s_j}) S_{n_j} \diag(g_{s_{j - 1} + 1}^{-1},\dots, g_{s_j}^{-1})\] for \(1\leq j \leq d\) and \[P_0 := \diag(g_1,\dots, g_{n_0})G_{n_0}(K, H) \diag(g_1^{-1},\dots, g_{n_0}^{-1}).\]
  Then \(P\) can be written as the direct product \(P_0\times P_1\times \cdots \times P_d\) with \(P_j\) acting on the vector space spanned by \(\{e_i\mid i\in I_j\}\).
  The group \(G_{n_0}(K, H)\) is normal in \(G_{n_0}(K, K) = K\wr S_{n_0}\), so \(P_0 = G_{n_0}(K, H)\).
  Define a map \(\xi: I\setminus I_0\to K\) via \(\xi(i) := g_{\sigma^{-1}(i)}\) for \(i \in I\setminus I_0\).
  This gives the desired equality \(P = P_{(I_0,\Pi, \xi)}\).
\end{proof}

As in \Cref{sec:reflarr}, write \(\mathcal P(G)\) for the poset of parabolic subgroups of \(G\) partially ordered by inclusion.
The ordering on the parabolic subgroups induces a partial order on the parabolic triples: \[(I_0, \Pi, \xi) \leq (I_0', \Pi', \xi') :\Longleftrightarrow P_{(I_0, \Pi, \xi)} \subseteq P_{(I_0', \Pi', \xi')}\;.\]

Notice that \(G_1(K, H) = H\).
In the following, we assume \(H\neq \{1\}\) so that trivial factors in the decomposition of a parabolic subgroup are always coming from a symmetric group, see also \Cref{rem:noH1}.

\begin{lemma}
  \label{lem:paraorder}
  Assume \(H\neq \{1\}\).
  Let \((I_0, \Pi, \xi)\) and \((I_0', \Pi', \xi')\) be parabolic triples.
  We have \((I_0, \Pi, \xi) \leq (I_0', \Pi', \xi')\) if and only if for every \(I_j'\in \Pi'\) there exist blocks \(I_{i_1}, \dots, I_{i_{p_j}}\in \Pi\) and \(g_1,\dots,g_{p_j}\in K\) such that \(I_j' = \bigcup_k I_{i_k}\) and for all \(l\in I_{i_k}\) we have \(\xi'(l) = g_{i_k}\xi(l)\).

  In particular, \(P_{(I_0, \Pi, \xi)} = P_{(I_0', \Pi', \xi')}\) if and only if \(I_0 = I_0'\), \(\Pi = \Pi'\) and there are \(g_1,\dots,g_{|{\Pi}|}\in K\) with \(\xi'(i) = g_j\xi(i)\) for every \(i\in I_j\).
\end{lemma}
\begin{proof}
  Assume \(P_{(I_0, \Pi, \xi)} \subseteq P_{(I_0', \Pi', \xi')}\).
  Then both groups maintain the block structure of the basis given by the respective partitions \((I_j)_j\) and \((I'_j)_j\).
  Further, if \(n_0 \neq 0\), the factor \(G_{n_0}(K, H)\) must be contained in \(G_{n_0'}(K, H)\) because \(H\neq \{1\}\).
  It follows that \(I_0 \subseteq I_0'\) and every \(I_j'\in \Pi'\) must be a union of blocks in \(\Pi\).
  Let \(I'_j\in\Pi'\) with \(I'_j = \bigcup_k I_{i_k}\) and let \(S_{I'_j}\) be the corresponding symmetric group permuting the vectors \(\{\xi'(l)e_l\mid l\in I'_j\}\).
  If we have \(|{I_{i_k}}| = 1\) for all \(k\), there is nothing to show.
  So, let \(|{I_{i_k}}| \geq 2\) and pick \(l_1,l_2\in I_{i_k}\).
  Then there is a permutation \(\sigma\in S_{I_{i_k}}\) that acts via \(\sigma(\xi(l_1)e_{l_1}) = \xi(l_2)e_{l_2}\).
  By assumption, \(\sigma\in S_{I'_j}\) as well, so \(\sigma(\xi'(l_1)e_{l_1}) = \xi'(l_2)e_{l_2}\).
  It follows that \(\xi^{-1}(l_1)\xi(l_2) = (\xi')^{-1}(l_1)\xi'(l_2)\), so \(\xi(l_1)(\xi')^{-1}(l_1) = \xi(l_2)(\xi')^{-1}(l_2)\).
  This must be fulfilled for arbitrary \(l_1,l_2\in I_{i_k}\), so \(g_{i_k}:= \xi(l_1)(\xi')^{-1}(l_1)\) is as desired.
  The same argument `read backwards' gives the claimed equivalence.
\end{proof}

Let \(\mathcal D_n(K)\) be the Dowling lattice of rank \(n\) corresponding to \(K\), see \cite{Dow73}.
The description of the partial ordering on parabolic triples in \Cref{lem:paraorder} gives an isomorphism between \(\mathcal P(G)\) and \(\mathcal D_n(K)\).
The analogue for imprimitive complex reflection groups is well-known, see \cite{Tay12}.
\begin{prop}
  \label{prop:dowling}
  Assume \(H\neq \{1\}\).
  There is an order preserving bijection between the poset \(\mathcal P(G_n(K, H))\) of parabolic subgroups of \(G_n(K, H)\) and the Dowling lattice \(\mathcal D_n(K)\).
\end{prop}
\begin{proof}
  Let \(P\leq G_n(K, H)\) be a parabolic subgroup.
  By Proposition~\ref{prop:paraclass}, there is a triple \((I_0, \Pi, \xi)\) with \(P = P_{(I_0, \Pi, \xi)}\).
  Then \(\Pi\) gives a partial partition of \(\{1,\dots,n\}\).
  Via the function \(\xi\), \(\Pi\) is turned into a partial \(K\)-partition, so we have established a map \(\mathcal P(G_n(K, H)) \to \mathcal D_n(K)\).
  Conversely, any partial \(K\)-partition gives a parabolic triple hence a parabolic subgroup.
  The latter assignment gives a well-defined map \(\mathcal D_n(K) \to \mathcal P(G_n(K, H))\) as two partial \(K\)-partitions are identified in \(\mathcal D_n(K)\) if and only if the corresponding parabolic subgroups are equal by Lemma~\ref{lem:paraorder}.
  The resulting bijection is order preserving by Lemma~\ref{lem:paraorder}.
\end{proof}

Our next result now follows from \Cref{prop:dowling} together with \Cref{cor:paratoint}.

\begin{corollary}
  Let \(K, H\leq \BBH^\times\) be finite groups with \([K, K]\leq H\trianglelefteq K\) and assume \(H\neq \{1\}\).
  There is an isomorphism of lattices \[L(\CA(G_n(K, H))) \cong \mathcal D_n(K).\]
\end{corollary}

\subsection{Orbit representatives}
\label{subsec:orbitreps}

Let \(G = G_n(K, H)\) be again  an imprimitive quaternionic reflection group for finite groups \(K, H\leq\BBH^\times\) and \(n\geq 3\).
We assume throughout that \(H \neq \{1\}\), see \Cref{rem:noH1}.
The group \(G\) acts naturally on the poset \(\mathcal P(G)\) by conjugation or, equivalently, on the lattice \(L(\CA(G))\) by linear transformations.
In the case of complex reflection groups (that is, \(K\) is cyclic), orbit representatives of this action are given in \cite[\S 3]{Tay12} for the parabolic subgroups and in \cite[\S 6.4]{orlikterao:arrangements} and \cite[\S 6.11]{douglaspfeifferroehrle:invariants} for the intersection lattice.
We extend these results to non-cyclic \(K\); the arguments are again largely analogous, however, there is one notable exception (\Cref{lem:noncyclicrep}).

We start by constructing a parabolic subgroup for any partial partition of \(n\).
Let \(m\in\{1,\dots, n\}\) and let \(\lambda = (\lambda_1,\dots,\lambda_k)\) be a partition of \(m\).
Put \(m_i := \sum_{j = 1}^i\lambda_j\) for \(1\leq i \leq k\) and \(m_0 := 0\).
Let \(I_0 := \{m + 1,\dots, n\}\) and \(I_i = \{m_{i - 1} + 1,\dots, m_i\}\) for \(1\leq i\leq k\) giving a partition \(\Pi = (I_1,\dots, I_d)\) of \(I\setminus I_0\).
Define the parabolic subgroup \(P_\lambda := P_{(I_0,\Pi,1)}\) where \(1\) denotes the map \(I\setminus I_0 \to K\) mapping everything to \(1 \in K\).

For \(\alpha\in K\), let \(\xi_\alpha:I\setminus I_0\to K\) be the map defined by \(\xi_\alpha(1) = \alpha\) and \(\xi_\alpha(i) = 1\) for \(i\geq 2\).
Set \(P_\lambda^\alpha := P_{(I_0, \Pi, \xi_\alpha)}\).
For a map \(\theta:\{1,\dots, n\}\to K\), write \(\hat{\theta}\in G_n(K, K)\) for the matrix with entries \(\theta(1),\dots,\theta(n)\) on the diagonal.

\begin{lemma}
  Let \(P\subseteq G_n(K, H)\) be a parabolic subgroup.
  Then there are an integer \(m\leq n\), a partition \(\lambda\) of \(m\) and \(\alpha\in K\) such that \(P\) is conjugate in \(G_n(K, H)\) to \(P_\lambda^\alpha\).
\end{lemma}
\begin{proof}
  By Proposition~\ref{prop:paraclass}, there is a parabolic triple \((I_0,\Pi,\xi)\) such that \(P = P_{(I_0, \Pi,\xi)}\).
  Put \(m := n - |{I_0}|\) and let \(\lambda\) be the partition of \(m\) coming from the (sorted) cardinalities of the blocks of \(\Pi\).
  After conjugating by a suitable permutation, we may assume that \(\Pi\) and \(I_0\) are given by \(\lambda\) as in the definition of the group \(P_\lambda\).
  Let \(\alpha = \xi(1)\cdots\xi(m)\in K\) and define \(\theta(1) = \alpha\xi(1)^{-1}\), \(\theta(i) = \xi(i)^{-1}\) for \(2\leq i\leq m\) and \(\theta(i) = 1\) for \(i\geq m + 1\).
  Because \(\theta(1)\cdots\theta(n) = \alpha\xi(1)^{-1}\cdots \xi(m)^{-1} = 1\in H\), we have \(\hat{\theta}\in G_n(K, H)\).
  So, \(P = \hat{\theta}^{-1}P_\lambda^\alpha\hat{\theta}\), as claimed.
\end{proof}

Hence, a system of orbit representatives of the action of \(G = G_n(K, H)\) on \(\mathcal P(G)\) can be chosen from the \(P_\lambda^\alpha\).
It remains to determine when two subgroups \(P_\lambda^\alpha\) and \(P_\mu^\beta\) are conjugate.
Clearly, this can only be the case, if \(\lambda = \mu\).

\begin{lemma}
  Let \(m < n\), \(\lambda = (\lambda_1,\dots,\lambda_k)\) a partition of \(m\) and \(\alpha\in K\).
  Then \(P_\lambda^\alpha\) is conjugate in \(G_n(K, H)\) to \(P_\lambda\).
\end{lemma}
\begin{proof}
  By assumption, we have \[P_\lambda^\alpha = S_{\lambda_1}\times\cdots\times S_{\lambda_k}\times G_{n - m}(K, H)\] with \(n - m > 0\).
  Conjugating by the matrix \(\diag(1,\dots,1, \alpha^{-1})\) leaves the group \(G_{n - m}(K, H)\) invariant.
  So, for \(\theta(1) = \alpha\), \(\theta(n) = \alpha^{-1}\) and \(\theta(i) = 1\), for \(2\leq i \leq n - 1\), we have \(\hat{\theta}P_\lambda\hat{\theta}^{-1} = P_\lambda^\alpha\) with \(\hat{\theta}\in G_n(K, H)\).
\end{proof}

By the classification of imprimitive quaternionic reflection groups from above, the quotient \(K/H\) fails to be cyclic only for \(K = \mathsf D_{2d}\) and \(H = \mathsf C_{2d}\) with \(d\geq 2\), see \Cref{tab:kleinian}.
In this case we have \(K/H \cong C_2\times C_2\).
We consider the cyclic case separately from this instance.

\begin{prop}
  Let \(\lambda = (\lambda_1,\dots,\lambda_k)\) be a partition of \(n\) and let \(\alpha,\beta\in K\).
  Assume that \(K/H\) is cyclic.
  Write \(d := \gcd([K:H], \lambda_1,\dots,\lambda_k)\).
  The group \(P_\lambda^\alpha\) is conjugate to \(P_\lambda^\beta\) in \(G_n(K, H)\) if and only if \[\ord_{K/H}(\alpha\beta^{-1})\;\left|\;\frac{[K:H]}{d}\right..\]
\end{prop}
\begin{proof}
  Without loss of generality, we may assume that \(\beta = 1\).
  Let \(n_i := \sum_{j = 1}^i\lambda_i\) for \(1\leq i\leq k\) and \(n_0:= 0\).
  Assume there is a \(g\in G_n(K, H)\) with \(gP_\lambda g^{-1} = P_\lambda^\alpha\).
  By multiplying \(g\) by a suitable permutation, we may assume that \(g\) is a diagonal matrix.
  So there is a map \(\theta:\{1,\dots, n\}\to K\) with \(\hat{\theta} = g\).
  Further, \(\hat{\theta}\) must maintain the block structure given by \(\lambda\) in the sense that there is a map \(\theta':\{1,\dots, n\}\to K\) which is constant on the sets \(\{n_{i - 1} + 1,\dots, n_i\}\) for \(1\leq i\leq k\), with \(\theta(1) = \alpha\theta'(1)\) and \(\theta(i) = \theta'(i)\) for \(i \geq 2\).
  Because \(\hat{\theta}\in G_n(K, H)\), we have \(\alpha\theta'(n_1)^{\lambda_1}\cdots\theta'(n_k)^{\lambda_k}\in H\).
  Let \(\zeta\in K\) be a generator of \(K/H\).
  Then there are \(s, s_1,\dots,s_k\in \BBZ_{\geq 0}\) with \(\alpha \equiv \zeta^s\) and \(\theta'(n_i) \equiv \zeta^{s_i}\) in \(K/H\).
  So, we have \[\zeta^s\zeta^{s_1\lambda_1 + \cdots + s_k\lambda_k} = \zeta^{t[K:H]}\] for some \(t\in \BBZ_{\geq 0}\).
  We have \(d\mid \lambda_i\) and \(d\mid [K:H]\), so \(d \mid s\).
  By choosing an appropriate \(\zeta\), we may assume that \(s = \frac{[K:H]}{\ord_{K/H}(\alpha)}\), giving the claim.

  Conversely, assume that there is an \(s\in \BBZ\) with \(s\cdot d\cdot\ord_{K/H}(\alpha) = [K:H]\).
  Let \(\zeta\in K\) be a generator of \(K/H\) with \(\zeta^{[K:H]/\ord_{K/H}(\alpha)} \equiv \alpha\) in \(K/H\).
  There are \(t, s_1,\dots,s_k\in \BBZ\) with \(d = t[K:H] + s_1\lambda_1 + \cdots + s_k\lambda_k\).
  We obtain \[\alpha \equiv \zeta^{st[K:H] + ss_1\lambda_1 + \cdots + ss_k\lambda_k}.\]
  Define \(\theta':\{1,\dots, n\}\to K\) by \(\theta'(j) := \zeta^{-ss_i}\) for \(n_{i - 1} + 1 \leq j \leq n_i\) and \(1\leq i \leq k\).
  From this, we obtain a map \(\theta:\{1,\dots, n\}\to K\) by setting \(\theta(1) := \alpha\theta'(1)\) and \(\theta(i) := \theta'(i)\) for \(i\geq 2\).
  Then \(\hat{\theta}\in G_n(K, H)\) because \[\theta(1)\cdots \theta(n) = \alpha\theta'(1)\cdots\theta'(n) \equiv 1\] in \(K/H\).
  By construction, \(\hat{\theta}\) leaves the block structure given by \(\lambda\) invariant, so we have \(\hat{\theta}P_\lambda\hat{\theta}^{-1} = P_\lambda^\alpha\) as required.
\end{proof}

\begin{lemma}
  \label{lem:noncyclicrep}
  Let \(K = \mathsf D_{2d}\) and \(H = \mathsf C_{2d}\) with \(d\geq 2\).
  Let \(\lambda = (\lambda_1,\dots,\lambda_k)\) be a partition of \(n\) and let \(\alpha,\beta\in K\).
  The group \(P_\lambda^\alpha\) is conjugate to \(P_\lambda^\beta\) in \(G_n(\mathsf D_{2d}, \mathsf C_{2d})\) if and only if \(\alpha\beta^{-1}\in H\) or \(2\nmid \gcd(\lambda_1,\dots,\lambda_k)\).
\end{lemma}
\begin{proof}
  The proof is similar to the one in the cyclic case.
  We may again assume that \(\beta = 1\).
  If \(\alpha\equiv 1\) in \(K/H\), there is nothing to prove.
  So, we have \(\ord_{K/H}(\alpha) = 2\) because \(K/H \cong C_2\times C_2\).

  If there is a \(g\in G_n(K,H)\) with \(gP_\lambda g^{-1} = P_\lambda^\alpha\), we can construct a map \(\theta':\{1,\dots, n\}\to K\) which is constant on the sets \(\{n_{i - 1} + 1,\dots, n_i\}\) and so that \(\theta\) defined by \(\theta(1):= \alpha\theta'(1)\) and \(\theta(i) := \theta'(i)\) for \(i\geq 2\) gives \(\hat{\theta}P_\lambda\hat{\theta}^{-1} = P_\lambda^\alpha\) as before.
  Then we have \[\alpha\theta'(n_1)^{\lambda_1}\cdots\theta'(n_k)^{\lambda_k} \in H,\] so there must be a \(1\leq j\leq k\) with \(2\nmid \lambda_j\) and hence \(2\nmid \gcd(\lambda_1,\dots,\lambda_k)\).

  If, conversely, \(2\nmid \gcd(\lambda_1,\dots,\lambda_k)\), then there is an \(i\) with \(2\nmid \lambda_i\).
  So, we may set \(\theta'(j) := \alpha\) for \(n_{i - 1} + 1\leq j\leq n_i\) and \(\theta'(j) := 1\) otherwise, and again \(\theta(1) := \alpha\theta'(1)\) and \(\theta(i) := \theta'(i)\) for \(i\geq 2\).
  Then \(\theta(1)\cdots \theta(n) = \alpha^{\lambda_i + 1}\equiv 1\) in \(K/H\), so \(\hat{\theta}\in G_n(K, H)\), and \(\hat{\theta}P_\lambda\hat{\theta}^{-1} = P_\lambda^\alpha\).
\end{proof}

We summarize the results above on the \(G_n(K, H)\)-conjugacy classes of parabolic subgroups  \(\mathcal P(G_n(K, H))\):

\begin{theorem}
  \label{thm:orbitreps}
  Representatives of the \(G_n(K, H)\)-conjugacy classes of parabolic subgroups  \(\mathcal P(G_n(K, H))\) are given as follows:
  \begin{enumerate}[(a)]
    \item For \(K/H\) cyclic:
      \begin{enumerate}[(i)]
        \item \(P_\lambda\) with \(\lambda\vdash m < n\);
        \item \(P_\lambda^{\alpha^s}\) with \(\lambda = (\lambda_1,\dots,\lambda_k)\vdash n\), \(\alpha\in K\) is a generator of \(K/H\), and \(0\leq s < \gcd([K:H], \lambda_1,\dots,\lambda_k)\);
      \end{enumerate}
    \item For \(K = \mathsf D_{2d}\) and \(H = \mathsf C_{2d}\):
      \begin{enumerate}[(i)]
        \item \(P_\lambda\) with \(\lambda\vdash m < n\);
        \item \(P_\lambda\) with \(\lambda = (\lambda_1,\dots,\lambda_k)\vdash n\) and \(2\nmid \gcd(\lambda_1,\dots,\lambda_k)\);
        \item \(P_\lambda^1\), \(P_\lambda^{\alpha}\), \(P_\lambda^{\beta}\), \(P_\lambda^{\gamma}\) with \(\lambda = (\lambda_1,\dots,\lambda_k)\vdash n\), \(2\mid \gcd(\lambda_1,\dots,\lambda_k)\), and \(\{1, \alpha,\beta,\gamma\}\subseteq K\) is a system of representatives of the residue classes of \(H\) in \(K\).
      \end{enumerate}
  \end{enumerate}
\end{theorem}

The orbit representatives in \(\mathcal P(G_n(K, H))\) given above correspond to orbit representatives in the lattice \(L(\CA(G_n(K, H)))\) by taking fixed spaces.
Concretely, if \(\{e_1,\dots,e_n\}\) is the standard basis of \(V = \BBH^n\), the fixed space of \(P_\lambda^\alpha\) is given by \[X_\lambda^\alpha := \langle \alpha e_1 + e_2 + \cdots + e_{n_1}, e_{n_1 + 1} + \cdots + e_{n_2}, \dots, e_{n_{k - 1} + 1} + \cdots + e_{n_k}\rangle\] with \(\lambda = (\lambda_1,\dots,\lambda_k)\) and \(n_i := \sum_{j = 1}^i\lambda_j\).

\section{Invariants of the imprimitive groups}
\label{sec:imprimgroups}
Let \(G = G_n(K, H)\) be an imprimitive quaternionic reflection group for finite groups \(K, H\leq \BBH^\times\) and \(n\geq 3\), where we continue to assume that \(H\neq \{1\}\).
Let \(\CA = \CA(G) = \CA_n(K)\) be the corresponding quaternionic reflection arrangement.
Recall that \(L(\CA)\) is endowed with a rank function \(\rk\) and we have \(\rk(X) = \codim_\BBH(X)\) for all \(X\in L(\CA)\).
As before, we write \(L(\CA)_k\) to denote the subset of elements of rank \(k\).
Let \(\mathcal X = \mathcal X(\CA, G)\) be a fixed set of orbit representatives of the action of \(G\) on \(L(\CA)\) and let \(\mathcal X(\CA, G)_k = \mathcal X(\CA, G)\cap L(\CA)_k\) be the representatives of a fixed rank \(k\).

We construct the Poincaré polynomial of \(H^*(M(\CA))^G\) inductively using \Cref{Gdecomp} and the identity \eqref{eq:euler}.
For this, we have to determine the set \[\mathcal X(\CA, G)_k^{\mathrm{tdi}} = \{X\in \mathcal X(\CA, G)_k \mid \dim H^{3k}(M(\CA_X))^{N_G(X)} \neq 0\}\] of orbit representatives that admit ``top degree invariants'', similarly to what is done in \cite{douglaspfeifferroehrle:invariants}.
As discussed in the previous section, \(X\in L(\CA)\) corresponds to a parabolic subgroup of \(G\) and hence to a partial partition of \(n\).
Write \(X_\lambda^\alpha\) for the fixed space of a parabolic subgroup \(P_\lambda^\alpha\leq G\).

The proof of the following lemma is analogous to the proof of \cite[Lem.~6.16]{douglaspfeifferroehrle:invariants}.

\begin{lemma}
  \label{lem:tdiparts}
  Let \(X_\lambda^\alpha \in \mathcal X(\CA, G)\) be an orbit representative.
  If \(X_\lambda^\alpha \in \mathcal X(\CA, G)^{\mathrm{tdi}}\), then \(\lambda \in\{\emptyset, (21^{m - 1}), (1^m)\}\) with \(m \geq 1\).
\end{lemma}

Notice that for a partition \(\lambda\), we have 
\[
\dim H^{3\rk(X_\lambda^\alpha)}(M(\CA_{X_\lambda^\alpha}))^{N_G(X_\lambda^\alpha)} = \dim H^{3\rk(X_\lambda)}(M(\CA_{X_\lambda}))^{N_G(X_\lambda)},
\] 
as \(X_\lambda^\alpha\) and \(X_\lambda\) only differ by a diagonal matrix in \(G_n(K, K)\).

We now study the fixed spaces corresponding to the partitions derived in \Cref{lem:tdiparts} in more detail.
As in \cite{douglaspfeifferroehrle:invariants}, 
let \[\eta_k = (1^{n - k})\text{, for }0\leq k\leq n - 1\text{, and }\tau_k = (21^{n - k - 1})\text{, for }1 \leq k \leq n - 1.\]
The only partitions of \(n\) are \(\eta_0\) and \(\tau_1\).
For \(n \geq 3\), each of the partitions hence index a unique orbit in \(L(\CA)\).
For \(n = 2\) and if \([K:H]\) is even, then the partition \(\tau_{n - 1} = (2)\) corresponds to two or four orbits.

Recall that the Coxeter group of type \(A_n\) for \(n \geq 0\) denotes the symmetric group \(S_{n + 1}\) acting on its irreducible \(n\)-dimensional representation.
In the following lemma, we write \(\CB_n^\BBH\) for the reflection arrangement of \(A_n\) considered as a quaternionic reflection group.
Further, we denote the image of the natural embedding of \(S_{n + 1}\) into \(\GL_{n + 1}(\BBH)\) by \(W_{n + 1}\).

\begin{lemma}
  \label{lem:tdipartsind}
  Let \(2\leq k \leq n - 1\).
  \begin{enumerate}[(1)]
    \item\label{lem:tdipartsind:1} \(H^{3k}(M(\CA_{X_{\eta_k}}))^{N_G(X_{\eta_k})} \cong H^{3k}(M(\CA_k(K)))^{G_k(K, K)}\).
    \item\label{lem:tdipartsind:2} If \(k < n - 1\), then \[H^{3k}(M(\CA_{X_{\tau_k}}))^{N_G(X_{\tau_k})} \cong H^{3k - 3}(M(\CA_{k - 1}(K)))^{G_{k - 1}(K, K)}\otimes H^3(M(\CB_2^\BBH)).\]
    \item\label{lem:tdipartsind:3} If \(k = n - 1\), then
      \begin{align*}
        H^{3k}(M(\CA_{X_{\tau_k}}))^{N_G(X_{\tau_k})} &= H^{3k}(M(\CA_{X_{\tau_k}}))^{Z_G(X_{\tau_k})} \\ &\cong H^{3n - 6}(M(\CA_{n - 2}(K)))^{G_{n - 2}(K, H)} \otimes H^3(M(\CB_2^\BBH)).
      \end{align*}
  \end{enumerate}
\end{lemma}
\begin{proof}
  The claims in \labelcref{lem:tdipartsind:1} and \labelcref{lem:tdipartsind:2} follow as in the proof of \cite[Lem.~6.18]{douglaspfeifferroehrle:invariants}.

  Let \(k = n - 1\) and \(\tau := \tau_k = (2)\).
  We have \(Z_G(X_\tau) = G_{n - 2}(K, H)\times W_2\) and \[H^{3n - 3}(M(\CA_{X_\tau}))^{Z_G(X_\tau)} \cong H^{3(n - 2)}(M(\CA_{n - 2}(K)))^{G_{n - 2}(K, H)} \otimes H^3(M(\CB_2^\BBH))^{W_2},\] as in \cite[6.15~(a)]{douglaspfeifferroehrle:invariants}.
  The normalizer \(N_G(X_\tau)\) consists of the block diagonal matrices \[\begin{pmatrix} d w_{n - 2} & \\ & e w_{2}\end{pmatrix}\] where \(w_j\in W_j\), \(d\in \GL_{n - 2}(\BBH)\) is a diagonal matrix with entries in \(K\), \(e\in \GL_2(\BBH)\) is a scalar matrix with entries in \(K\), and \(d_{1,1}\cdots d_{n - 2, n - 2}\cdot e_{1,1}\cdot e_{2,2}\in H\).
  Hence, \(N_G(X_\tau)\) acts by scalars on \(\CB_2^\BBH\) and we conclude \[H^{3k}(M(\CA_{X_{\tau_k}}))^{N_G(X_{\tau_k})} = H^{3k}(M(\CA_{X_{\tau_k}}))^{Z_G(X_{\tau_k})}\] 
  which finishes the proof of \labelcref{lem:tdipartsind:3}.
\end{proof}

To be able to use \Cref{lem:tdipartsind} for inductive arguments, we first need to consider the cases \(n \leq 2\).
For \(n = 0\), \(G\) is the trivial group and we have \(\dim H^0(M(\CA))^G = 1\).
For \(n = 1\), \(\CA\) consists of a single hyperplane on which \(G\) acts trivially, so we have \(\dim H^3(M(\CA))^G = 1\) as well.

\begin{lemma}
  \label{lem:n2}
  We have
  \[\dim H^6(M(\CA_2(K)))^{G_2(K, H)} = \begin{cases}
      2,& \text{if }2\mid [K:H] \text{ and } K/H\text{ is cyclic,}\\
      4,& \text{if }2\mid [K:H] \text{ and } K/H\text{ is not cyclic,}\\
      1,& \text{otherwise.}
    \end{cases}\]
\end{lemma}
\begin{proof}
  Let \(G = G_2(K, H)\).
  By \Cref{prop:dim2}, we only need to determine the cardinality of \(\mathcal X(\CA, G)_1\).
  The representatives \(\mathcal X(\CA, G)_1\) are labelled by the partial partitions \((2)\) and \((1)\) of \(n = 2\).
  The partial partition \((1)\) corresponds to a unique orbit representative, but \((2)\) does in general not.
  With the orbit representatives given in \Cref{thm:orbitreps}, we have
  \[|{\mathcal X(\CA, G)_1}| =
    \begin{cases}
      3,& \text{if }2\mid [K:H] \text{ and } K/H\text{ is cyclic,}\\
      5,& \text{if }2\mid [K:H]  \text{ and } K/H\text{ is not cyclic,}\\
      2,& \text{otherwise,}
    \end{cases}\]
  giving the claim.
\end{proof}

We are now prepared for the main theorem of this section.
Besides the labelling by partitions, we may also label the elements of \(\mathcal X(\CA, G)_k^{\mathrm{tdi}}\) by their reflection type, that is, their labelling in the classification \cite{cohen:quaternionic}.
Note that \(G_n(K, H)\) with \(K\) cyclic corresponds to a complex reflection group, so this case is covered by \Cref{thm:poincarecompred}.

\begin{table}
  \caption{\(\dim H^{3\rk(X)}(M(\CA_X))^{N_G(X)}\) for \((\CA, G_n(K, H))\), \(n\geq 3\), \(K\) not cyclic}
  \label{tab:tdi}
  \begin{tabular}{ccccc}
    \hline
    \multicolumn{2}{l}{rank} & \(k = 0\) & \multicolumn{2}{c}{\(1 \leq k \leq n - 2\)}\\
    \hline
    \multicolumn{2}{l}{partition} & \((1^n)\) & \((21^{n - k - 1})\) & \((1^{n - k})\)\\
    \multicolumn{2}{l}{reflection type} & \(A_0\) & \(G_{k - 1}(K, H)A_1\) & \(G_k(K, H)\)\\
    \hline
    \multirow{2}{*}{\([K:H]\) and \(n\) even} & \(K/H\) cyclic & 1 & 1 & 1\\
     & else & 1 & 1 & 1\\
     else & & 1 & 1 & 1\\
     \hline
  \end{tabular}
  \vskip1em
  \begin{tabular}{ccccc}
    \hline
    \multicolumn{2}{l}{rank} & \multicolumn{2}{c}{\(k = n - 1\)} & \(k = n\)\\
    \hline
    \multicolumn{2}{l}{partition} & \((2)\) & \((1)\) & \(\emptyset\) \\
    \multicolumn{2}{l}{reflection type} & \(G_{n - 2}(K, H)A_1\) & \(G_{n - 1}(K, H)\) & \(G_n(K, H)\) \\
    \hline
    \multirow{2}{*}{\([K:H]\) and  \(n\) even} & \(K/H\) cyclic & 2 & 1 & 2 \\
     & else & 4 & 1 & 4 \\
     else & & 1 & 1 & 1 \\
    \hline
  \end{tabular}
\end{table}

\begin{theorem}
  \label{thm:poincareimprim}
  Let \(G = G_n(K, H)\) be an imprimitive irreducible quaternionic reflection group with \(n\geq 3\) and assume that \(K\) is not cyclic.
  \Cref{tab:tdi} lists the elements of \(\mathcal X(\CA, G)^{\mathrm{tdi}}\) via their corresponding partitions and reflection types and the dimensions \(\dim H^{3\rk(X)}(M(\CA_X))^{N_G(X)}\) for \(X\in\mathcal X(\CA, G)^{\mathrm{tdi}}\).
\end{theorem}
\begin{proof}
  The information regarding the partitions and reflection types in \Cref{tab:tdi} follows from \Cref{lem:tdiparts}.
  We verify the dimensions given in \Cref{tab:tdi}.
  For \(k = 0\), we have \(X_{\eta_0} = V\) and \(Z_G(X_{\eta_0})\) is the trivial group, so indeed \(\dim H^0(M(\CA_V)) = 1\).
  For \(k = 1\), the arrangements \(\CA_{X_{\lambda}}\) for \(\lambda\in\{\tau_1,\eta_1\}\) both consist of a single hyperplane, on which the normalizer \(N_G(X_\lambda)\) acts trivially.
  Hence \(\dim H^3(M(\CA_{X_\lambda})) = 1\) in both cases.
  Let \(2\leq k \leq n - 2\).
  Then we have \[\dim H^{3k}(M(\CA_{X_{\eta_k}}))^{N_G(X_{\eta_k})} = \dim H^{3k}(M(\CA_k(K)))^{G_k(K, K)}\] and \[\dim H^{3k}(M(\CA_{X_{\tau_k}}))^{N_G(X_{\tau_k})} = \dim H^{3k - 3}(M(\CA_{k - 1}(K)))^{G_{k - 1}(K, K)}\] by \Cref{lem:tdipartsind}.
  This verifies the entries of the columns labelled \(1 \leq k \leq n - 2\) of \Cref{tab:tdi} by induction because we have \([K:K] = 1\).

  Let \(k = n - 1\).
  We can argue as in the previous case for \(\eta_{n - 1}\).
  For \(\tau_{n - 1}\), \Cref{lem:tdipartsind} gives \[\dim H^{3k}(M(\CA_{X_{\tau_k}}))^{N_G(X_{\tau_k})} = \dim H^{3n - 6}(M(\CA_{n - 2}(K)))^{G_{k - 2}(K, H)}.\]
  So the entries in this column of the table again follow by induction.

  The entries of the last column follow using \eqref{eq:euler} and what has been proved so far.
\end{proof}

Recall that the cohomology of \(M(\CA)\) only lives in degrees divisible by 3.
To increase readability, we present the Poincaré polynomials in the following corollary evaluated at \(t^{1/3}\).

\begin{corollary}
  \label{cor:poincareimprim}
  Let \(G = G_n(K, H)\) be an imprimitive quaternionic reflection group with \(n\geq 2\) and assume that \(K\) is not cyclic.
  \begin{enumerate}[(1)]
    \item\label{cor:poincareimprim:1} If both \([K:H]\) and \(n\) are even and \(K/H\) is cyclic, then \[P(\CA(G), G;t^{1/3}) = 1 + 2t + \cdots + 2t^{n - 2} + 3t^{n - 1} + 2t^n.\]
    \item\label{cor:poincareimprim:2} If both \([K:H]\) and \(n\) are even and \(K/H\) is not cyclic, then \[P(\CA(G), G;t^{1/3}) = 1 + 2t + \cdots + 2t^{n - 2} + 5t^{n - 1} + 4t^n.\]
    \item\label{cor:poincareimprim:3} If \([K:H]\) or \(n\) are odd, then \[P(\CA(G), G;t^{1/3}) = 1 + 2t + \cdots + 2t^{n - 1} + t^n.\]
  \end{enumerate}
\end{corollary}
\begin{proof}
  For \(n = 2\), this follows from \Cref{lem:n2} and \Cref{prop:dim2}.
  For \(n \geq 3\), we may use \Cref{tab:tdi} together with \[\dim H^{3k}(M(\CA))^G = \dim\left(\bigoplus_{X\in\mathcal X(\CA, G)_k}H^{3k}(M(\CA_X))^{N_G(X)}\right).\]
  Recall that every partition in the table corresponds to a unique element of \(\mathcal X(\CA, G)\).
\end{proof}

\section{Invariants of the primitive groups}
\label{sec:primgroups}

It remains to determine the Poincaré polynomials for the primitive irreducible quaternionic reflection groups.
As discussed in \Cref{sec:quatreflgroups}, almost all of these groups act on a vector space of quaternionic dimension \(n = 2\), so are covered by \Cref{prop:dim2}.
There are precisely seven groups in dimension higher than 2 which are labelled \(W(Q)\), \(W(R)\), \(W(S_1)\), \(W(S_2)\), \(W(S_3)\), \(W(T)\) and \(W(U)\) in \cite{cohen:quaternionic}.
\Cref{tab:primgroups} lists the Poincaré polynomials \(P(\CA(G), G; t)(t^{1/3})\) for these groups.
These polynomials were computed using the computer algebra system \textsf{OSCAR} \cite{Dec+25,Osc25}.

\begin{table}
  \caption{Poincaré polynomials of \(H^*(M(\CA(G)))^G\) for primitive irreducible quaternionic reflection groups \(G\) in dimension \(n > 2\).}
  \label{tab:primgroups}
  \begin{tabular}{p{4em}p{2em}l}
    \hline
    \(G\) & \(n\) & \(P(\CA(G), G;t^{1/3})\) \\
    \hline
    \(W(Q)\) & 3 & \(1 + t\)\\
    \(W(R)\) & 3 & \(1 + t\)\\
    \(W(S_1)\) & 4 & \(1 + t + t^3 + t^4\)\\
    \(W(S_2)\) & 4 & \(1 + t + t^3 + t^4\)\\
    \(W(S_3)\) & 4 & \(1 + t + t^3 + t^4\)\\
    \(W(T)\) & 4 & \(1 + t + t^3 + t^4\)\\
    \(W(U)\) & 5 & \(1 + t + t^4 + t^5\)\\
    \hline
  \end{tabular}
\end{table}

For these computations, we used the matrix generators of the groups one obtains from the root systems given in \cite{cohen:quaternionic}, see also \cite[\S 7]{BST23}.
For each group \(G\) with reflection arrangement \(\CA = \CA(G)\), we constructed a vector space basis of the corresponding Orlik--Solomon algebra \(H^*(M(\CA))\) via a \emph{non-broken circuit basis} using \cite[Algorithm~NBC]{BDPR13}.
Any homogeneous element of \(H^*(M(\CA))\) can be efficiently written in this basis with the algorithm given in \cite[Proof of Theorem~2.5]{CE01}.
The action of \(G\) on \(H^*(M(\CA))\) is linear, so every component \(H^k(M(\CA))\) is a representation of \(G\).
We can now explicitly construct matrices in \(\GL(H^k(M(\CA)))\) corresponding to the action of \(G\) on the non-broken circuit basis.
This allows us to determine the character \(\chi_k\) of the representation \(H^k(M(\CA))\).
The \(k\)-th coefficient of the Poincaré polynomial of \(H^*(M(\CA))^G\) is then the scalar product of \(\chi_k\) with the trivial character of \(G\).

\section{Bases for \(H^*(M(\CA))^G\)}
\label{sec:bases}
We close with a discussion of bases of \(H^*(M(\CA(G)))^G\), analogous to \cite[\S 7]{douglaspfeifferroehrle:invariants}.
If \(\dim(V) = 2\), bases are given in \Cref{prop:dim2}.
Bases in the complex reducible case are constructed in \cite{douglaspfeifferroehrle:invariants}.

\subsection{The imprimitive groups}

Let \(G = G_n(K, H)\) be an imprimitive quaternionic reflection group for finite groups \(K, H\leq \BBH^\times\) and \(n\geq 3\).
As before, we assume that \(H\neq \{1\}\), see \Cref{rem:noH1}.
Let \(\CA = \CA(G) = \CA_n(K)\) be the corresponding quaternionic reflection arrangement.
We consider the hyperplanes \(H_1 = \ker(x_1)\) and \(H_i = \ker(x_{i - 1} - x_i)\) for \(2\leq i \leq n\) in \(\CA\).
To improve readability, we write \(h_i = e_{H_i}\) for the generators in \(H^3(M(\CA))\).
For \(\alpha\in K\), we further put \(H_2^\alpha = \ker(x_1 - \alpha x_2)\in \CA\) and \(h_2^\alpha\in H^3(M(\CA))\).
Notice \(h_2^1 = h_2\).

We construct bases of \(H^*(M(\CA))\) inductively.
For this, we only need to consider those parabolic subgroups of \(G\) with fixed space \(X\in \mathcal X(\CA, G)^{\mathrm{tdi}}\).
Write \(\CT(\CA, G)^{\mathrm{tdi}}\) for the set of reflection types of these groups, that is, we have \[\CT(\CA, G)^{\mathrm{tdi}} = \{A_0, A_1\}\cup \{G_k(K, H)\mid 1\leq k\leq n\} \cup \{G_{k - 1}(K, H)A_1\mid 2\leq k \leq n - 1\}\] by \Cref{thm:poincareimprim}.
Recall that for \(n \geq 3\) every element of \(\CT(\CA, G)^{\mathrm{tdi}}\) corresponds uniquely to an orbit representative \(\mathcal X(\CA, G)^{\mathrm{tdi}}\).

For the cases \(T \in\{A_0, A_1\}\), we put \(B_{A_0}^G = \{1\}\) and \(B_{A_1}^G = \{h_2\}\).
We define
\begin{align*}
  b_T^{G,\alpha} = \begin{cases}
    h_1h_2^\alpha h_3\cdots h_k,& \text{if }T = G_k(K, H)\text{ and }1\leq k \leq n,\\
    h_1h_2^\alpha h_3\cdots h_{k - 1}h_{k + 1}, & \text{if }T = G_{k - 1}(K, H)A_1\text{ and }2\leq k \leq n - 1,
  \end{cases}
\end{align*}
where we again omit the symbol \(\wedge\) for the products in \(H^*(M(\CA))\).
If \(T = G_k(K, H)\) with \(1\leq k \leq n - 1\) or \(T = G_{k - 1}(K, H)A_1\) with \(2\leq k\leq n - 2\), set \(B_T^G = \{b_T^{G, 1}\}\).
For \(T \in \{G_n(K, H), G_{n - 2}(K, H)A_1\}\), we distinguish the following cases.
\begin{itemize}
  \item If \([K:H]\) or \(n\) is odd, put \(B_T^G = \{b_T^{G, 1}\}\).
  \item If \([K:H]\) and \(n\) are even and \(K/H\) is cyclic, let \(\alpha\in K\) be a generator of \(K/H\) and put \(B_T^G = \{b_T^{G, 1}, b_T^{G, \alpha}\}\).
  \item If \([K:H]\) and \(n\) are even and \(K/H\) is not cyclic, let \(1, \alpha, \beta,\gamma\in K\) be a system of representatives of the residue classes in \(K/H\).
    Put \(B_T^G = \{b_T^{G, 1}, b_T^{G, \alpha}, b_T^{G,\beta}, b_T^{G,\gamma}\}\).
\end{itemize}

For \(T\in \CT(\CA, G)^{\mathrm{tdi}}\), we write \(X_T\in\mathcal X(\CA, G)^{\mathrm{tdi}}\) for the corresponding orbit representative.
\begin{theorem}
  \label{thm:basisimprim}
  Let \(G = G_n(K, H)\), \(n\geq 3\), be an irreducible imprimitive quaternionic reflection group with reflection arrangement \(\CA\). Assume that \(K\) is not cyclic.
  For \(k\geq 0\), let \(\CT(\CA, G)_k^{\mathrm{tdi}}\) be the set of reflection types of rank \(k\) with top degree invariants.
  \begin{enumerate}[(1)]
    \item\label{thm:basisimprim:1} For \(T\in \CT(\CA, G)^{\mathrm{tdi}}\), the set \(\epsilon_{N_G(X_T)}\cdot B_T^G\) is a basis of \(H^{\rk X_T}(M(\CA_{X_T}))^{N_G(X_T)}\).
    \item\label{thm:basisimprim:2} For \(k\geq 0\), the disjoint union \(\coprod_{T\in\CT(\CA, G)_k^{\mathrm{tdi}}} \epsilon_G \cdot B_T^G\) is a basis of \(H^k(M(\CA))^G\).
  \end{enumerate}
\end{theorem}
\begin{proof}
  Part \labelcref{thm:basisimprim:2} follows from \labelcref{thm:basisimprim:1} using the identity \[H^{3k}(M(\CA))^G = \sum_{T\in \mathcal T(\CA, G)_k^{\mathrm{tdi}}} \epsilon_G\cdot H^{3k}(M(\CA_{X_T}))^{N_G(X_T)}\] from \Cref{Gdecomp} \labelcref{Gdecomp:2}.

  To prove part \labelcref{thm:basisimprim:1}, we consider the possibilities for \(T\) according to \Cref{thm:poincareimprim}.
  For \(T\in\{A_0, A_1, G_1(K, H)\}\), the claim is clear.
  Let \(T = G_k(K, H)\) with \(2\leq k \leq n - 1\).
  Then \(\dim(H^{3k}(M(\CA_{X_T}))^{N_G(X_T)}) = 1\) and we need to show that \(\epsilon_{N_G(X_T)} \cdot b_T^{G,1}\neq 0\).
  The partial partition corresponding to \(T\) is \((1^{n - k})\), so by \Cref{lem:tdipartsind} \labelcref{lem:tdipartsind:1}, we have an isomorphism \(H^{3k}(M(\CA_{X_T}))^{N_G(X_T)} \cong H^{3k}(M(\CA_k(K)))^{G_k(K, K)}\).
  This isomorphism is induced from the isomorphism \(\CA_{X_T} \cong \CA_k(K)\) and hence sends \(h_1h_2\cdots h_k\) in \(H^{3k}(M(\CA_{X_T}))\) to \(h_1h_2\cdots h_k\) in \(H^{3k}(M(\CA_k(K)))\).
  Then \(\epsilon_{N_G(X_T)}\cdot h_1h_2\cdots h_k\) is sent to \(\epsilon_{G_k(K, K)}\cdot h_1h_2\cdots h_k\) because the isomorphism maintains the different group actions.
  By induction, \(H^{3k}(M(\CA_k(K)))^{G_k(K, K)}\) is 1-dimensional with basis \(\{\epsilon_{G_k(K, K)}\cdot h_1h_2\cdots h_k\}\).
  We conclude \(\epsilon_{N_G(X_T)}\cdot h_1h_2\cdots h_k\neq 0\) and hence \(\{\epsilon_{N_G(X_T)}\cdot b_T^{G, 1}\}\) is a basis for \(H^{3k}(M(\CA_{X_T}))^{N_G(X_T)}\) as claimed.

  Let \(T = G_{k - 1}(K, H)A_1\) with \(2\leq k\leq n - 2\).
  Then again \(\dim(H^{3k}(M(\CA_{X_T}))^{N_G(X_T)}) = 1\) and we need to show that \(\epsilon_{N_G(X_T)} b_T^G\neq 0\).
  This follows analogously to the previous case using the isomorphism in \Cref{lem:tdipartsind} \labelcref{lem:tdipartsind:2}.

  For \(T = G_{n - 2}(K, H)A_1\), we use the isomorphism \[H^{3n - 3}(M(\CA_{X_T}))^{N_G(X_T)} \cong H^{3n - 6}(M(\CA_{n - 2}(K)))^{G_{n - 2}(K, H)}\otimes H^3(M(\CB_2^\BBH))\] from \Cref{lem:tdipartsind} \labelcref{lem:tdipartsind:3}.
  By induction and \Cref{prop:dim2}, we have that \(\epsilon_{G_{n - 2}(K, H)}\cdot B_{G_{n - 2}(K, H)}^{G_{n - 2}(K, H)}\) is a basis for \(H^{3n - 6}(M(\CA_{n - 2}(K)))^{G_{n - 2}(K, H)}\).
  The desired basis of \(H^{3n - 3}(M(\CA_{X_T}))^{N_G(X_T)}\) follows with the given isomorphism.

  Finally, let \(T = G_n(K, H)\).
  We assume that \(\dim(H^{3n}(M(\CA))^G) = 4\) and in particular \(n \geq 4\), the other cases follow analogously.
  We need to show that the set \[\epsilon_G\cdot B_G^G = \{\epsilon_G\cdot h_1h_2h_3\cdots h_n, \epsilon_G\cdot h_1h_2^\alpha h_3\cdots h_n,\epsilon_G\cdot h_1h_2^\beta h_3\cdots h_n,\epsilon_G\cdot h_1h_2^\gamma h_3\cdots h_n\}\] with \(1,\alpha,\beta,\gamma\in K\) a system of representatives of \(K/H\) is a basis for \(H^{3n}(M(\CA))^G\).
  We apply the derivation \(\partial\) introduced in \Cref{sec:invariants} to these elements:
  \[\partial(\epsilon_G\cdot h_1h_2^\xi h_3\cdots h_n) = \epsilon_G\cdot h_2^\xi h_3\cdots h_n - \epsilon_G\cdot h_1 h_3\cdots h_n + \sum_{i = 3}^n(-1)^{i - 1}\epsilon_G\cdot h_1h_2^\xi h_3\cdots \widehat{h_i}\cdots h_n,\] where \(\xi\in\{1,\alpha,\beta,\gamma\}\).
  Consider the element \(h_2^\xi h_3\cdots h_n\).
  The pointwise stabilizer of the intersection \(X = H_2^\xi\cap\cdots\cap H_n\) is of type \(A_{n - 2}\).
  Hence \[H^{3n - 3}(M(\CA_X))^{Z_G(X)} \cong H^{3n - 3}(M(\CB_{n - 2}^\BBH))^{W_{n - 1}} = 0\] by \cite[Thm.~7]{Bri73} and therefore \(\epsilon_G \cdot h_2^\xi h_3\cdots h_n = 0\).
  Similarly, one sees \(\epsilon_G \cdot h_1h_3\cdots h_n = 0\) and \(\epsilon_G \cdot h_1h_2^\xi h_3\cdots \widehat{h_i}\cdots h_n = 0\) for \(3\leq i \leq n - 2\) because these elements correspond to the reflection types \(G_1(K, H)A_{n - 2}\) and \(G_{i - 1}(K, H)A_{n - i}\), respectively, which do not have top degree invariants for \(n \geq 4\).
  We are left with \[\partial(\epsilon_G \cdot h_1h_2^\xi h_3\cdots h_n) = \epsilon_G \cdot h_1h_2^\xi h_3\cdots h_{n - 2}h_n - \epsilon_G \cdot h_1 h_2^\xi h_3\cdots h_{n - 1}\in H^{3n - 3}(M(\CA))^G.\]
  By the results so far, we have the basis \[\{\epsilon_G\cdot h_1h_2\cdots h_{n - 1},\epsilon_G\cdot h_1h_2^\xi h_3\cdots h_{n - 2} h_n \mid \xi = 1,\alpha, \beta,\gamma\}\] of \(H^{3n - 3}(M(\CA))^G\).
  The elements \(\epsilon_G\cdot h_1h_2^\xi h_3\cdots h_{n - 1}\) must be multiples of \(\epsilon_G\cdot h_1h_2\cdots h_{n - 1}\) because \(\dim(H^{3n - 3}(M(\CA_{n - 1}(K)))^{G_{n - 1}(K, H)} = 1\).
  Therefore the elements \(\partial(\epsilon_G \cdot h_1h_2^\xi h_3\cdots h_n)\) with \(\xi\in\{1, \alpha,\beta,\gamma\}\) are linearly independent in \(H^{3n - 3}(M(\CA))^G\).
  It follows that \(\epsilon_G \cdot B_G^G\) is linearly independent in \(H^{3n}(M(\CA))^G\) and hence is a basis.
\end{proof}

\begin{remark}
  \label{rem:nogens}
  We should emphasize that the hyperplanes \(H_1,\dots, H_n\) and \(H_2^\alpha\) do not correspond to a set of generators of the group \(G_n(K, H)\) unlike in the complex case discussed in \cite{douglaspfeifferroehrle:invariants}.
  Indeed, the hyperplane \(H_1\) is the fixed space of any element \(\diag(g, 1,\dots, 1)\) with \(g\in H\setminus\{1\}\).
  But the group \(H\) is in general not cyclic, so in a set of generators of \(G_n(K, H)\) there might be two non-redundant generators with fixed space \(H_1\).
\end{remark}

\subsection{The primitive groups}
Let \(G\leq\GL(V)\) be one of the seven primitive irreducible quaternionic reflection groups, which we discussed in \Cref{sec:primgroups}.
Let \(\CA = \CA(G)\) be the corresponding reflection arrangement and \(n = \dim(V)\).
According to \Cref{tab:primgroups}, we have \(\dim(H^{3k}(M(\CA))^G) = 1\) for \(k = 0, 1\).
Bases for these degrees are given by \(\{1\}\) and \(\{\epsilon_G\cdot h\}\), respectively, where \(h\) corresponds to some hyperplane \(H\in \CA\).

Assume in the following that \(G\) has top degree invariants, so \(G\notin\{W(Q), W(R)\}\).
To complete the bases for \(H^*(M(\CA))^G\), we need to find a non-zero element in degree \(3(n - 1)\) and a non-zero element in degree \(3n\) by \Cref{tab:primgroups} again.
The contribution in degree \(3(n - 1)\) comes from a parabolic subgroup \(P\leq G\) of rank \(n - 1\) with top degree invariants.
Combining the results from \cite{douglaspfeifferroehrle:invariants} and this article with the lists of parabolic subgroups in \cite[\S 7.2]{BST23}, we can identify these groups.

\Cref{tab:basisprim} lists the groups \(G\) together with the parabolic subgroups \(P\) and the hyperplanes that give a basis of \(H^{3(n - 1)}(M(\CA))^G\) and \(H^{3n}(M(\CA))^G\).
The data in the table are to be interpreted as follows.
The hyperplanes are given by their linear forms, that is, if \(f\) is a polynomial in the table, then \(\ker(f)\) is the corresponding hyperplane.
If \(f_1,\dots,f_n\) are the polynomials listed for the group \(G\), then these give generators \(h_1,\dots, h_n\in H^3(M(\CA))\) corresponding to the hyperplanes \(\ker(f_1),\dots, \ker(f_n)\).
Bases of \(H^{3(n - 1)}(M(\CA))^G\) and \(H^{3n}(M(\CA))^G\) are then given by \(\{\epsilon_G\cdot h_1\cdots h_{n - 1}\}\) and \(\{\epsilon_G \cdot h_1\cdots h_n\}\), respectively.

The results in \Cref{tab:basisprim} were computed using \textsf{OSCAR} \cite{Dec+25,Osc25}.
For the groups \(W(S_2)\) and \(W(T)\), the reflections fixing the hyperplanes given in the table generate the corresponding groups.
For the other groups \(G\), this is not the case and there is no basis of \(H^{3n}(M(\CA))^G\) that is related to generators of \(G\), similar to \Cref{rem:nogens}.

\begin{table}
  \caption{Hyperplanes giving a basis of \(H^{3n}(M(\CA))^G\)}
  \label{tab:basisprim}
  \renewcommand{\arraystretch}{1.2}
  \begin{tabular}{llll}
    \hline
    \(G\) & \(P\) & \(n\) & Linear forms of hyperplanes\\
    \hline
    \(W(S_1)\) & \(\mathsf C_2\times \mathsf C_2\times \mathsf C_2\) & 4 &  \(x_1 - x_4\), \(x_2 + x_3\), \(x_2 - x_3\), \(x_1 + x_3\)\\
    \hline
    \(W(S_2)\) & \(G(2, 1, 3)\) & 4 & \(x_1\), \(x_2 - x_3\), \(x_2 + x_3\), \(x_1 - \mathbf j x_2 - \mathbf i x_3 - \mathbf k x_4\)\\
    \hline
    \(W(S_3)\) & \(G_3(\mathsf D_2, \mathsf C_2)\) & 4 &\(x_2 -\mathbf i x_3\), \(x_2 - \mathbf k x_3\), \(x_1 + \mathbf k x_2 + \mathbf j x_3 + \mathbf i x_4\), \(x_1 + x_3\)\\
    \hline
    \(W(T)\) & \(H_3\) & 4 & \(2x_1 -(\psi \mathbf i + \phi \mathbf j - \mathbf k)x_2 -(\phi \mathbf i - \mathbf j - \psi \mathbf k) x_3 + (\mathbf i + \psi \mathbf j + \phi \mathbf k)x_4\),\\
    & & & \(2x_1 - (\mathbf i + \psi\mathbf j - \phi\mathbf k)x_2 + (\psi \mathbf i + \phi\mathbf j + \mathbf k)x_3 - (-\phi \mathbf i + \mathbf j - \psi\mathbf k) x_4\), \\
    & & & \(2x_1 - (-\phi \mathbf i + \mathbf j - \psi\mathbf k) x_2 + (\mathbf i + \psi \mathbf j - \phi \mathbf k) x_3 - (\psi \mathbf i + \phi \mathbf j + \mathbf k)x_4\), \\
    & & & \(-\phi x_1 + \psi x_2 + x_3\) \\
    \hline
    \(W(U)\) & \(W(S_1)\) & 5 & \(x_1 + \frac{1}{2}(1 + \mathbf i + \mathbf j + \mathbf k) x_2 + \frac{1}{2}(1 + \mathbf i + \mathbf j + \mathbf k)x_3 + x_4\), \\
    & & & \(x_1 - \frac{1}{2}(1 + \mathbf i + \mathbf j + \mathbf k) x_2 - \frac{1}{2}(1 + \mathbf i + \mathbf j + \mathbf k)x_3 + x_4\),\\
    & & & \(x_5\), \\
    & & & \(x_1 + x_3 + \frac{1}{2}(1 + \mathbf i + \mathbf j + \mathbf k)x_4 -\frac{1}{2}(1 + \mathbf i + \mathbf j + \mathbf k)x_5\), \\
    & & & \(x_1 + x_2 + \frac{1}{2}(1 - \mathbf i - \mathbf j - \mathbf k)x_3 - \frac{1}{2}(1 - \mathbf i - \mathbf j - \mathbf k)x_5\) \\
    \hline
  \end{tabular}
  \vskip .5ex
  {\raggedright Abbreviations: \(\phi = \frac{1 + \sqrt 5}{2}\), \(\psi = \frac{1 - \sqrt{5}}{2}\)\par}
\end{table}

\end{document}